\documentclass[a4paper,12pt]{amsart}
\usepackage{amsmath,amsfonts,amssymb,amsthm,latexsym,amscd}
\usepackage[dvips]{graphicx}
\usepackage[T1]{fontenc}

\usepackage{parskip}

\newcommand{\B}[1]{{\mathbf #1}}
\newcommand{\C}[1]{{\mathcal #1}}

\newtheorem{thm}{Theorem}[section]
\newtheorem{thm*}{Theorem}
\newtheorem{lem}[thm]{Lemma}
\newtheorem{prop}[thm]{Proposition}
\newtheorem{cor}[thm]{Corollary}

\newtheorem*{q*}{Question}

\theoremstyle{definition}

\newtheorem{rem}[thm]{Remark}
\newtheorem*{rem*}{Remark}
\newtheorem*{rems*}{Remarks}
\newtheorem*{cor*}{Corollary}

\def\B{\mathbf}

\newcommand\Diff{\operatorname{Diff}}

\newcommand\Ham{\operatorname{Ham}}

\newcommand\Aut{\operatorname{Aut}}

\newcommand\vol{\operatorname{vol}}
\newcommand\OP{\operatorname}

\def\C{\mathcal{C}}
\def\D{\mathbf{D}}
\def\Diff{\operatorname{Diff}}
\def\G{\mathcal{G}}

\def\cH{\mathcal{H}}
\def\Id{\operatorname{Id}}
\def\Im{\operatorname{Im}}

\def\Frag{\operatorname{Frag}}

\def\X{\operatorname{X}}

\def\a{\alpha}
\def\b{\beta}
\def\area{\operatorname{area}}

\def\ev{\operatorname{ev}}
\def\emb{\operatorname{emb}}

\def\g{\gamma}
\def\o{\omega}
\def\s{\sigma}

\begin{document}

\title[Bi-invariant metrics and quasi-morphisms]{Bi-invariant metrics and quasi-morphisms on groups of Hamiltonian diffeomorphisms of surfaces}
\author{Michael Brandenbursky}

\subjclass[2010]{Primary 53; Secondary 57}

\begin{abstract}
Let $\B\Sigma_g$ be a closed orientable surface of genus $g$ and let
$\Diff_0(\B\Sigma_g,\area)$ be the identity component of the group of area-preserving diffeomorphisms of $\B\Sigma_g$. In this work we present the extension of Gambaudo-Ghys construction to the case of a closed hyperbolic surface $\B\Sigma_g$, i.e. we show that every non-trivial homogeneous quasi-morphism on the braid group on $n$ strings of $\B\Sigma_g$ defines a non-trivial homogeneous quasi-morphism on the group $\Diff_0(\B\Sigma_g,\area)$. As a consequence we give another proof of the fact that the space of homogeneous quasi-morphisms on $\Diff_0(\B\Sigma_g,\area)$ is infinite dimensional.

Let $\Ham(\B\Sigma_g)$ be the group of Hamiltonian diffeomorphisms of $\B\Sigma_g$. As an application of the above construction we construct two injective homomorphisms $\B Z^m\to\Ham(\B\Sigma_g)$, which are bi-Lipschitz with respect to the word metric on $\B Z^m$ and the autonomous and fragmentation metrics on $\Ham(\B\Sigma_g)$. In addition, we construct a new infinite family of Calabi quasi-morphisms on $\Ham(\B\Sigma_g)$.
\end{abstract}

\maketitle

\section{Introduction}

\label{S:intro}

Let $\B\Sigma_g$ be a closed orientable surface of genus $g$ equipped with the area form, and let $\Diff_0(\B\Sigma_g,\area)$ be the identity component of the group of area-preserving diffeomorphisms of $\B\Sigma_g$. Denote by $\D^2$ and $\B S^2$ the unit $2$-disc in a plane and the standard $2$-sphere respectively. In \cite{GG} Gambaudo and Ghys gave a construction of quasi-morphisms on $\Diff(\D^2,\area)$ and on $\Diff(\B S^2,\area)$ from quasi-morphisms on the Artin pure braid group $\B P_n$ and on the pure sphere braid group $\B P_n(\B S^2)$ respectively. In this paper we are going to extend their construction to the case of $\Diff_0(\B\Sigma_g,\area)$, when $g>1$, and to the case of $\Ham(\B\Sigma_g)$, when $g=1$. Here $\Ham(\B\Sigma_g)$ is the group of Hamiltonian diffeomorphisms of $\B\Sigma_g$.  More precisely, we are going to show that every non-trivial homogeneous quasi-morphism on the full surface braid group $\B B_n(\B\Sigma_g)$ defines a non-trivial homogeneous quasi-morphism on the group $\Diff_0(\B\Sigma_g,\area)$.

As a consequence of this construction, we construct two injective homomorphisms $\B Z^m\to\Ham(\B\Sigma_g)$, which are bi-Lipschitz with respect to the word metric on $\B Z^m$ and the autonomous and fragmentation metrics on $\Ham(\B\Sigma_g)$ respectively. In particular, we generalize the results of the author and Kedra proved for the case of $\Diff(\D^2,\area)$, see \cite{BK2}.

A quasi-morphism on $\Ham(\B\Sigma_g)$ is called Calabi if it equals to the Calabi homomorphism on every diffeomorphism supported in a displaceable disc. They have various applications in dynamics and symplectic geometry, see e.g. \cite{BEP,EP}. Calabi quasi-morphisms on $\Ham(\B \Sigma_g)$ were constructed by Entov-Polterovich \cite{EP} (in case when $g=0$) and by Py \cite{Py,Py1} (in case when $g\geq1$). In this paper we give a new topological construction of infinitely many linearly independent Calabi quasi-morphisms on $\Ham(\B\Sigma_g)$ for every $g>1$.

\subsection{Generalized Gambaudo-Ghys construction and main results}

\label{SS:GG+results}

Let us start with a definition.
A function $\psi\colon \Gamma \to~\B R$ from a group $\Gamma$ to the reals is called a {\em quasi-morphism}
if there exists a real number $A\geq 0$ such that
$$
|\psi(fh) - \psi(f) - \psi(h)|\leq A
$$
for all $f,h\in \Gamma$. The infimum of such $A$'s is called the
\emph{defect} of $\psi$ and is denoted by $D_\psi$.
If $\psi(f^n)=n\psi(f)$ for all $n\in \B Z$
and all $g\in \Gamma$ then $\psi$ is called \emph{homogeneous}. Any
quasi-morphism $\psi$ can be homogenized by setting
$$\overline{\psi} (f) := \lim_{p\to +\infty} \frac{\psi (f^p)}{p}.$$
The vector space of homogeneous quasi-morphisms on $\Gamma$
is denoted by $Q(\Gamma)$. The space of homogeneous quasi-morphisms on $\Gamma$ modulo the space of homomorphisms on $\Gamma$ is denoted by $\widehat{Q}(\Gamma)$. For more information about quasi-morphisms and their connections to different brunches of mathematics, see \cite{Calegari}.

For every pair of points $x,y\in\B\Sigma_g$ let us choose a geodesic path $s_{xy}\colon[0,1]\to \B\Sigma_g$ from $x$ to $y$.
Let $f_t\in \Diff_0(\B\Sigma_g,\area)$ be an isotopy from the identity to $f\in\Diff_0(\B\Sigma_g,\area)$
and let $z\in \B\Sigma_g$ be a basepoint. For $y\in\B\Sigma_g$ we define a loop
$\gamma_{y}\colon [0,1]\to\B\Sigma_g$ by
\begin{equation}\label{eq:gamma-for-qm}
\gamma_{y}(t):=
\begin{cases}
s_{zy}(3t) &\text{ for } t\in \left [0,\frac13\right ]\\
f_{3t-1}(y) &\text{ for } t\in \left [\frac13,\frac23\right ]\\
s_{f(y)z}(3t-2) & \text{ for } t\in \left [\frac23,1\right ].
\end{cases}
\end{equation}

Let $\X_n(\B\Sigma_g)$ be the configuration space of all ordered $n$-tuples
of pairwise distinct points in the surface $\B\Sigma_g$. It's fundamental group
$\pi_1(\X_n(\B\Sigma_g))$ is identified with the pure surface braid group $\B P_n(\B\Sigma_g)$.
Let point $z=(z_1,\ldots,z_n)$ in $\X_n(\B\Sigma_g)$ be a base point. For almost each point $x=(x_1,\ldots,x_n)\in\X_n(\B\Sigma_g)$ the $n$-tuple of loops $(\gamma_{x_1},\ldots,\gamma_{x_n})$ is
a based loop in the configuration space $\X_n(\B\Sigma_g)$.

Let $g>1$. Then the group $\Diff_0(\B\Sigma_g,\area)$ is simply-connected
(see e.g. \cite[Section 7.2]{P-book}). Hence the based homotopy
class of this loop does not depend on the choice of the
isotopy $f_t$. Let $\gamma(f,x)\in \B P_n(\B\Sigma_g)$
be an element represented by this loop.

Let $g=1$. In this case the group $\Diff_0(\B\Sigma_1,\area)$ is not simply-connected. However, the group $\Ham(\B\Sigma_1)$ is simply-connected (see e.g. \cite[Section 7.2]{P-book}). Hence, for every Hamiltonian isotopy $f_t$ between $\Id$ and $f$, the based homotopy
class of the loop $(\gamma_{x_1},\ldots,\gamma_{x_n})$ does not depend on the choice of the isotopy $f_t$.

Let $\varphi\colon \B B_n(\B\Sigma_g)\to \B R$ be a homogeneous quasi-morphism, where $\B B_n(\B\Sigma_g)$ is a full braid group on $n$ strings on a surface $\B\Sigma_g$. For $g>1$ we define the map
$$\Phi_n\colon \Diff_0(\B\Sigma_g,\area)\to \B R$$
by
\begin{equation}\label{eq:GG-ext}
\Phi_n(f):=
\int\limits_{\X_n(\B\Sigma_g)}\varphi(\g(f;{x}))d{x}\qquad\qquad
\overline{\Phi}_n(f):=\lim_{p\to +\infty}\frac{\Phi_n(f^p)}{p}
\thinspace .
\end{equation}
For $g=1$ the map $\Phi_n\colon \Ham(\B\Sigma_1)\to \B R$ is defined as in \eqref{eq:GG-ext}.

\begin{thm*}\label{T:gen-GG}
Let $g>1$. Then the function $\overline{\Phi}_n\colon\Diff_0(\B\Sigma_g,\area)\to~\B R$ is a well defined homogeneous quasi-morphism.
Let $g=1$. Then the function $\overline{\Phi}_n\colon\Ham(\B\Sigma_1)\to\B R$ is a well defined homogeneous quasi-morphism.
\end{thm*}

\begin{rem}
Note that when $n=1$ the group $\B B_1(\B\Sigma_g)=\pi_1(\B\Sigma_g)$. In this case Theorem \ref{T:gen-GG} was proved by
L. Polterovich in \cite{P} and by Gambaudo-Ghys in \cite{GG}. We must add that when $g=1$ and $n=1$ the Polterovich quasi-morphism $\overline{\Phi}_1$ becomes trivial on $\Ham(\B\Sigma_1)$.
\end{rem}

By Theorem \ref{T:gen-GG}, in case when $g>1$, the above construction defines linear map
$$\G_{n,g}\colon Q(\B B_n(\B\Sigma_g))\to Q(\Diff_0(\B\Sigma_g,\area)),$$
and in case when $g=1$ it defines linear map
$$\G_{n,1}\colon \widehat{Q}(\B B_n(\B\Sigma_1))\to Q(\Ham(\B\Sigma_1)).$$

\begin{thm*}\label{T:prop-GG-map}
Let $g\neq1$. The map $\G_{n,g}\colon Q(\B B_n(\B\Sigma_g))\to Q(\Diff_0(\B\Sigma_g,\area))$ is injective for each $n\in\B N$.
\end{thm*}

\begin{rem}
In case when $g=1$, one can similarly show that the map $\G_{n,1}\colon \widehat{Q}(\B B_n(\B\Sigma_1))\to Q(\Ham(\B\Sigma_1))$ is injective.
In case when $g=0$, Theorem \ref{T:prop-GG-map} was proved by Ishida~\cite{I}.
\end{rem}

\subsection{Applications}
\label{SS:Applications}

\subsubsection{Autonomous metric}

Recall that $\B\Sigma_g$ is a compact orientable surface of genus $g$. Let $H\colon\B\Sigma_g\to \B R$ be a smooth function with zero mean.
There exists a unique vector field $X_H$ such that
$$\iota_{X_H}\o=dH.$$
It is easy to see that $X_H$ is tangent to the level sets of $H$. Let $h$ be the
time-one map of the flow $h_t$ generated by $X_H$. The diffeomorphism $h$ is
area-preserving and every diffeomorphism arising in this way is called
{\em autonomous}. Such a diffeomorphism is relatively easy to understand
in terms of its generating function.

Let $\Ham(\B\Sigma_g)$ be the group of Hamiltonian diffeomorphisms of $\B\Sigma_g$. Since it is a simple group \cite{Ba}, every Hamiltonian diffeomorphism of $\B\Sigma_g$ is a composition of finitely many autonomous diffeomorphisms. The set of such diffeomorphisms is invariant under conjugation. Hence we define the conjugation-invariant {\em autonomous norm} on the group $\Ham(\B\Sigma_g)$ by
$$
\|f\|_{\OP{Aut}}:=
\min\left\{m\in \B N\,|\,f=h_1\circ\cdots\circ h_m \right\},
$$
where each $h_i$ is autonomous.
The associated metric is defined by
$${\B d}_{\OP{Aut}}(f,h):=\|fh^{-1}\|_{\OP{Aut}}.$$
Since the autonomous norm is conjugation-invariant, the autonomous metric is bi-invariant.

\begin{thm*}\label{T:main}
Let $g>1$. Then for every natural number $m\in \B N$ there exists an
injective homomorphism
$$T\colon\B Z^m\hookrightarrow\Ham(\B\Sigma_g)$$
which is bi-Lipschitz with respect to the word metric on $\B Z^m$ and the autonomous metric on
$\Ham(\B\Sigma_g)$.
\end{thm*}

\textbf{Remarks.}
\begin{enumerate}
\item
The study of bi-invariant metrics on groups of geometric origin was initiated by
Burago, Ivanov and Polterovich \cite{BIP}. The autonomous metric is a particular case of such metrics. Moreover, it is the most natural word metric on the group of Hamiltonian diffeomorphisms of a symplectic manifold. It is particulary interesting in the two-dimensional case due to the following observation.

Let $H\colon \B\Sigma_g\to \B R$ be a Morse function and let $h$ be a Hamiltonian diffeomorphism generated by $H$. After cutting the surface $\B\Sigma_g$ along critical level sets we are left with finite number of regions, so that each one of them is diffeomorphic to annulus. By Arnol'd-Liouville theorem \cite{Arnold}, there exist angle-action symplectic coordinates on each one of these annuli, so that $h$ rotates each point on a regular level curve with the same speed, i.e. the speed depends only on the level curve. It follows that a generic Hamiltonian diffeomorphism of $\B\Sigma_g$ may be written as a finite composition of autonomous diffeomorphisms, such that each one of these diffeomorphisms is "almost everywhere rotation" in right coordinates, and hence relatively simple. Of course this decomposition is nor unique, neither canonical. However, it is plausible that it might be useful in dynamical systems.

\item
Let $\Diff(\B D^2,\area)$ be the group of smooth compactly supported
area-preserving diffeomorphisms of the open unit disc $\B D^2$ in the Euclidean plane. In \cite{BK2} together with Kedra we showed that for every positive number $k\in \B N$ there exists an injective homomorphism $\B Z^k\to \Diff(\B D^2,\area)$ which is bi-Lipschitz
with respect to the word metric on $\B Z^k$ and the autonomous metric on
$\Diff(\B D^2,\area)$.
We should mention that Gambaudo and Ghys showed that the diameter of the group of area-preserving diffeomorphisms of the 2-sphere equipped with the autonomous metric is infinite as well, see \cite[Section 6.3]{GG}.

\item
Let $(\B M, \o)$ be a closed symplectic manifold. It is interesting to know whether the autonomous metric is unbounded on the group of Hamiltonian diffeomorphisms $\Ham(\B M)$, and if yes then can one prove Theorem \ref{T:main} in the case when the dimension of
$\B M$ is greater than two?
\end{enumerate}

\subsubsection{Fragmentation metric}

Let $(\B M,\o)$ be a closed symplectic manifold, and denote by $\Ham(\B M)$ the group of Hamiltonian diffeomorphisms of $\B M$. Let $\{\D_\a\}$ be a collection of regions such that
\begin{itemize}
\item For each $\a$ the region $\D_\a\subset \B M$ is diffeomorphic to a ball of the same dimension as $\B M$.

\item All elements in $\{\D_\a\}$ have the same volume $r$. Every region of volume $r$ in $\B M$, which is diffeomorphic to
a $\dim(\B M)$-dimensional ball, lies in $\{\D_\a\}$.
\end{itemize}
Since $\Ham(\B M)$ is a simple group \cite{Ba}, every diffeomorphism in $\Ham(\B M)$ is a composition of finitely many diffeomorphisms, such that each such diffeomorphism is supported in some ball $\D_\a$. The set of such diffeomorphisms is invariant under conjugation. We define the conjugation-invariant {\em fragmentation norm} on the group $\Ham(\B M)$ with respect to the cover $\{\D_\a\}$ by
$$
\|f\|_{\OP{Frag}}:=
\min\left\{m\in \B N\,|\,f=h_1\circ\cdots\circ h_m \right\},
$$
where the support of $h_i$ lies in some $\D_{\a}$.
The associated metric is defined by
${\B d}_{\OP{Frag}}(f,h):=\|fh^{-1}\|_{\OP{Frag}}$. Since the fragmentation norm is conjugation-invariant, the fragmentation metric is bi-invariant.

\begin{rem}
Recently, fragmentation metric received a considerable attention, since, in particular, it is related to the question about the simplicity of the group of compactly supported area-preserving homeomorphisms of an open 2-disc, see e.g. \cite{EPP,LR}.
\end{rem}

\begin{thm*}\label{T:frag-metric}
\textbf{1.} Let $\B\Sigma_g$ be any closed surface of genus $g>1$. Then there exists an
injective homomorphism
$$I\colon\B Z^{2g-2}\hookrightarrow\Ham(\B\Sigma_g)$$
which is bi-Lipschitz with respect to the word metric on $\B Z^{2g-2}$ and the fragmentation metric on
$\Ham(\B\Sigma_g)$.\\
\textbf{2.} Let $(\B M,\o)$ be a closed symplectic manifold of dimension greater then two. Suppose that there exists an embedding of the free group on two generators $\B F_2\hookrightarrow\pi_1(\B M)$ such that the image of the induced map $\widehat{Q}(\pi_1(\B M))\to \widehat{Q}(\B F_2)$ is infinite dimensional. Then for every natural number $m\in \B N$ there exists an injective homomorphism
$$I\colon\B Z^m\hookrightarrow\Ham(\B M)$$
which is bi-Lipschitz with respect to the word metric on $\B Z^m$ and the fragmentation metric on
$\Ham(\B M)$.
\end{thm*}

It is known that every non-elementary word-hyperbolic group $\Gamma$ contains $\B F_2$ such that the induced map
$\widehat{Q}(\Gamma)\to \widehat{Q}(\B F_2)$ is infinite dimensional, see \cite{HO}.
\begin{cor}
Let $(\B M,\o)$ be a closed negatively curved symplectic manifold of dimension greater then two. Then $(\Ham(\B M), \B d_{\OP{Frag}})$ contains bi-Lipschitz embedded finitely generated free abelian group of an arbitrary rank.
\end{cor}

\subsubsection{Calabi quasi-morphisms on $\Ham(\B\Sigma_g)$}
\label{subsec-Calabi-qm}

Let $(\B M,\o)$ be a symplectic manifold. Recall that a homogeneous quasi-morphism $\Psi$ on $\Ham(\B M)$ is called Calabi if its restriction to a subgroup $\Ham(\B B)$, where $\B B$ is a displaceable ball in $\B M$, coincides with Calabi homomorphism on $\Ham(\B B)$, see e.g. \cite{EP,BEP,Ost}.

Here we focus on the case of surfaces, but the works quoted above deal also with groups of Hamiltonian
diffeomorphisms of higher dimensional symplectic manifolds. In \cite{EP} Entov and Polterovich constructed a Calabi quasi-morphism on a group of Hamiltonian diffeomorphisms of the two-sphere. Their construction uses quantum homology. They asked whether one can construct a Calabi quasi-morphism on $\Ham(\B \Sigma_g)$, where $g\geq 1$. Around 2006 Pierre Py gave a positive answer to this question, see \cite{Py,Py1}.

Let $g>1$. In this paper we give a different construction of an infinite family of Calabi quasi-morphisms on
$\Ham(\B \Sigma_g)$.

Let $\D$ be an open disc in the Euclidean plane and denote by $\Ham(\D)$ the group of compactly supported Hamiltonian diffeomorphisms of $\D$.
Let $\varphi\colon \B B_2\to \B R$ be a homomorphism, which takes value $1$ on the generator of the braid group $\B B_2$ on two generators, which is cyclic. Let $\overline{\Phi}_2$ be the induced, by Gambaudo-Ghys construction, homomorphism from $\Ham(\D)$ to the reals. The Calabi homomorphism
$\C_{\D}\colon\Ham(\D)\to\B R$ may be defined as follows:
\begin{equation}\label{eq:Calabi-disc}
\C_{\D}(h):=\overline{\Phi}_2(h).
\end{equation}
For the proof of this fact see e.g. \cite{GG1}. It turns out that the generalized Gambaudo-Ghys construction allows us to construct an infinite family of linearly independent Calabi quasi-morphisms, i.e. we prove the following

\begin{thm*}\label{T:Calabi-qm}
Let $g>1$. The space of Calabi quasi-morphisms on $\Ham(\B\Sigma_g)$ is infinite dimensional.
\end{thm*}

\begin{rem}
We should note that the above theorem is not entirely new, because Py construction \cite{Py} implies a proof of this theorem as well. Nevertheless, we think that this theorem is interesting, since our construction is entirely different from the one given by Pierre Py.
\end{rem}

\section{Proofs}
\label{sec-proofs}

\subsection{Proof of Theorem \ref{T:gen-GG}}

It is enough to prove that the functions
$$\Phi_n\colon \Diff_0(\B\Sigma_g,\area)\to \B R\qquad \Phi_n\colon \Ham(\B\Sigma_1)\to \B R$$
are well defined quasi-morphisms for each $n>1$.

Let $g>1$. In \cite{Cal} E. Calabi constructed a homomorphism
$$\C_g\colon\Diff_0(\B\Sigma_g)\to H_1(\B\Sigma_g,\B R)\cong\B R^{2g}.$$
Banyaga proved that the kernel of $\C_g$ is a simple group, see \cite{Ba}. Let us recall a useful definition of the Calabi homomorphism.
For each $1\leq i\leq n$ the fundamental group $\pi_1(\B\Sigma_g)$ is generated by homotopy classes of $2g$ simple closed curves such that
\begin{equation}\label{eq:Pi-1-presentation}
\pi_1(\B\Sigma_g)=\langle\a_i,\b_i|\hspace{2mm} 1\leq i\leq g,\thinspace \prod_{i=1}^g [\a_i,\b_i]=1\rangle .
\end{equation}
Let $\Pi_{\a_i}\colon\pi_1(\B\Sigma_g)\to\B Z$ and $\Pi_{\b_i}\colon\pi_1(\B\Sigma_g)\to\B Z$ be homomorphisms, where $\Pi_{\a_i}(w)$ equals to the number of times $\a_i$ appears in $w$ minus number of times $\a_i^{-1}$ appears in $w$, and $\Pi_{\b_i}(w)$ equals to the number of times $\b_i$ appears in $w$ minus number of times $\b_i^{-1}$ appears in $w$. Then
$$\C_g\colon\Diff_0(\B\Sigma_g,\area)\to H_1(\B\Sigma_g,\B R)\cong\B R^{2g}$$
may be defined as follows:
$$
\C_g(h)_i:=
\begin{cases}
\lim\limits_{p\to\infty}\frac{1}{p}\int\limits_{\B\Sigma_g}\Pi_{\a_i}(\g(h^p;{x}))d{x}\text{ for } 1\leq i\leq g\\
\lim\limits_{p\to\infty}\frac{1}{p}\int\limits_{\B\Sigma_g}\Pi_{\b_{i-g}}(\g(h^p;{x}))d{x}\text{ for } g+1\leq i\leq 2g,
\end{cases}
$$
where $\g(h^p;{x})$ is defined in ~\eqref{eq:gamma-for-qm}, and $\C_g(h)_i$ is the $i$-th coordinate of $\C_g(h)$.

Let $f\in\Diff_0(\B\Sigma_g,\area)$. It follows that $f$ can be written as a product $f=f'\circ h_1\circ\ldots\circ h_{2g}$, where $f'\in\ker(\C_g)$ and each $h_i$ is a time-one map of the isotopy $\{h_{t,i}\}\in\Diff_0(\B\Sigma_g,\area)$ which is defined as follows:

\begin{itemize}
\item
For each $1\leq i\leq g$ and $t\in \B R$ the diffeomorphism $h_{t,i}$ equals to the identity
outside some tubular neighborhood of the curve $\a_i$, that is identified with $[0,1]\times\B S^1$, and
the time-one map $h_i$ is equal to the identity on $\a_i$, which is identified with $\left\{\frac{1}{2}\right\}\times\B S^1$.
\item
For each $g+1\leq i\leq 2g$ and $t\in \B R$ the diffeomorphism $h_{t,i}$ equals to the identity
outside some tubular neighborhood of the curve $\b_{i-g}$, that is identified with $[0,1]\times\B S^1$, and
the time-one map $h_i$ is equal to the identity on $\b_{i-g}$, which is identified with $\left\{\frac{1}{2}\right\}\times\B S^1$.
\item
Each diffeomorphism $h_{t,i}$ preserves the foliation of
$[0,1]\times\B S^1$ by the circles $\{x\}\times \B S^1$. Each diffeomorphism $h_{t,i}$ preserves the orientation for $t\geq 0$.
For all $s,t\in\B R$ we have $h_{i,t+s}=h_{i,t}\circ h_{i,s}$.
\item
In addition we require that $\C_g(f)=\C_g(h_1\circ\ldots\circ h_{2g})$.
\end{itemize}

It was proved by Banyaga that there exists $m\in\B N$ and a family $\{f_i\}_{i=1}^m$ of diffeomorphisms in the group $\Diff_0(\B\Sigma_g,\area)$, such that each $f_i$ is supported in some disc $\B D_i\subset\B\Sigma_g$ and $f'=f_1\circ\ldots\circ f_m$, see \cite{Ba1}. For each
$1\leq i\leq m$ pick an isotopy $\{f_{t,i}\}$ in $\Diff_0(\B\Sigma_g,\area)$ between the identity and $f_i$, such that the support of $\{f_{t,i}\}$ lies in the same disc $\B D_i$. First we are going to show that $|\Phi_n(f_i)|<\infty$.

Let $\X_n(\B D_i)$ be the configuration space of all ordered $n$-tuples
of pairwise distinct points in the subsurface $\B D_i\subset\B\Sigma_g$.
Since changing a basepoint in $\X_n(\B\Sigma_g)$ changes $\Phi_n$ by a bounded value, we can assume that the basepoint $z$ lies in $\X_n(\B D_i)$. Theorems 2 and 4 in \cite{Bir} imply that the inclusion of $\B D_i$ into $\B\Sigma_g$ induces an inclusion of the Artin braid group $\B B_n$ into the surface braid group $\B B_n(\B\Sigma_g)$. For $1\leq j<n$ the group $\B B_j$ may be viewed as a subgroup of $\B B_n$ by adding $n-j$ strings. Since $\varphi\colon \B B_n(\B\Sigma_g)\to \B R$ is invariant under conjugation and the support of the isotopy $\{f_{t,i}\}$ lies in $\B D_i$ we have
\begin{eqnarray*}
\Phi_n(f_i)&=&\int\limits_{\X_n(\B\Sigma_g)}\varphi(\g(f_i;{x}))d{x}\\
&=&\sum_{j=1}^n\dbinom{n}{j}\vol(\X_{n-j}(\B\Sigma_g\setminus \B D_i))\int\limits_{\X_j(\B D_i)}\varphi(\g(f_i;{x}))d{x}.
\end{eqnarray*}
Since the integral $\int\limits_{\X_j(\B D_i)}\varphi(\g(f_i;{x}))d{x}$ is well defined for each $1\leq j\leq n$ \cite[Lemma 4.1]{B}, we have $|\Phi_n(f_i)|<\infty$.

Denote by $\mathfrak{A}_i$ the tubular neighborhood of $\a_i$ diffeomorphic to $[0,1]\times\B S^1$ if $1\leq i\leq g$ and of $\b_{i-g}$ if $g+1\leq i\leq 2g$. By definition
\begin{equation*}
\Phi_n(h_i):=\int\limits_{\X_n(\B\Sigma_g)}\varphi(\g(h_i;{x}))d{x}\thinspace .
\end{equation*}
Since $h_i$ rotates all circles of $\mathfrak{A}_i$ and is identity on $\B\Sigma_g\setminus\mathfrak{A}_i$, it is easy to compute $\Phi_n(h_i)$. The computation is very similar to one done by Gambaudo-Ghys in \cite[Section 5.2]{GG}, and hence we omit it. In particular, it follows that $|\Phi_n(h_i)|<\infty$.

Let $g',h'\in \Diff_0(\B\Sigma_g,\area)$. Then
\begin{eqnarray*}
&&|\Phi_n(g'h')-\Phi_n(g')-\Phi_n(h')|\\
&\leq&\int\limits_{\X_n(\B\Sigma_g)}|\varphi(\g(g'h';{x}))-\varphi(\g(g';h'(x)))-\varphi(\g(h';{x}))|d{x}\\
&\leq&\vol(\X_n(\B\Sigma_g))\cdot D_{\varphi}\thinspace,
\end{eqnarray*}
i.e. $\Phi_n$ satisfies the quasi-morphism condition. It follows that
$$|\Phi_n(f)|\leq (m+2g)\left(D_{\varphi}\cdot\vol(\X_n(\B\Sigma_g))+\sum_{i=1}^m|\Phi_n(f_i)|+\sum_{i=1}^{2g}|\Phi_n(h_i)|\right).$$
Hence $\Phi_n$ is a well defined quasi-morphism.

Let $g=1$ and $f\in\Ham(\B\Sigma_1)$. Since the group $\Ham(\B\Sigma_1)$ is simple, there exists $m\in\B N$ and a family $\{f_i\}_{i=1}^m$ of diffeomorphisms in $\Ham(\B\Sigma_1)$, such that each $f_i$ is supported in some disc $\B D_i\subset\B\Sigma_1$ and $f=f_1\circ\ldots\circ f_m$, see \cite{Ba1}. For each
$1\leq i\leq m$ we pick an isotopy $\{f_{t,i}\}$ in $\Ham(\B\Sigma_1)$ between the identity and $f_i$, such that the support of $\{f_{t,i}\}$ lies in the same disc $\B D_i$. Now we proceed as in the case when $g>1$.
\qed

\subsection{Proof of Theorem \ref{T:prop-GG-map}}
\label{ssec-proof-GG-properties}

Let $g>1$. Recall that the group $\pi_1(\B\Sigma_g, z_i)$ is generated by homotopy classes of $2g$ simple closed curves such that
\begin{equation}\label{eq:Pi-1-i-presentation}
\pi_1(\B\Sigma_g, z_i)=\langle\a_{ij},\b_{ij}|\hspace{2mm} 1\leq j\leq g,\thinspace \prod_{j=1}^g [\a_{ij},\b_{ij}]=1\rangle .
\end{equation}
At this point we describe a particular generating set $\mathcal{S}$ for the pure surface braid group $\B P_n(\B \Sigma_g)$. In \cite{Bel} Bellingeri showed that $\B P_n(\B \Sigma_g)$ is generated by the set
$$\mathcal{S}=\{A_{i,j}\thinspace|\hspace{2mm} 1\leq i\leq 2g+n-1,\thinspace 2g+1\leq j\leq 2g+n,\thinspace i<j\}.$$
We view the Artin pure braid group $\B P_n$ as a subgroup of $\B P_n(\B \Sigma_g)$. Each generator in the set
$$\mathcal{S}'=\{A_{i,j}\thinspace|\thinspace 2g+1\leq i<j\leq 2g+n\}$$
is the corresponding band generator of $\B P_n$, i.e. it twists only $i$-th and $j$-th strands. Each generator in the set $\mathcal{S}\setminus\mathcal{S}'$ naturally corresponds to one of the generators of $\pi_1(\B\Sigma_g, z_i)$ described above, see \cite{Bel}.

Let us start the proof.
In case when $g=0$ this theorem was proved by Ishida in \cite{I}. In what follows we adopt the proof of Ishida to the general case of $g>1$. Let
$0\neq\phi\in Q(\B B_n(\B\Sigma_g))$. It is sufficient to prove that $\G_{n,g}(\phi)\neq 0$. Since $\phi$ is a non-trivial homogeneous quasi-morphism, there exists a braid $\b\in \B P_n(\B\Sigma_g)$ such that $\phi(\b)\neq 0$. The braid $\b$ is pure, hence it can be written as a composition of generators and their inverses from the set $\mathcal{S}$.

Let $\{U_i\}_{i=1}^n$ be a family of disjoint subsets of $\B\Sigma_g$, such that $z_i\in U_i$ and each $U_i$ is diffeomorphic to a disc. For any pair $(U_i,U_j)$  we take subsets $W\subset V\subset\B\Sigma_g$ such that $W$ and $V$ are diffeomorphic to a disc, $U_i\cup U_j\subset W$ and $V\cap U_k=\emptyset$ for $k\neq i,j$. Let $\{h_t\}$ be a path in $\Diff_0(\B\Sigma_g,\area)$ which fixes the outside and a small neighborhood of $V$ and rotates $W$ once. Such $\{h_t\}$ twists all $U_i$'s in the form of a generator $A_{i+2g,j+2g}\in\mathcal{S}'$.

In addition, for each $U_i$
we take $2g$ subsets $W_{i,j}\subset V_{i,j}\subset\B\Sigma_g$ such that both $W_{i,j}$ and $V_{i,j}$ are diffeomorphic to $\B S^1\times[0,1]$ and $U_i\subset W_{i,j}$ for each $1\leq j\leq 2g$. We identify $W_{i,j}$ with $\B S^1\times[0,1]$, and require that for each $1\leq j\leq 2g$ the curve $\B S^1\times\{\frac{1}{2}\}$ in $W_{i,j}$ represents a unique generator from the presentation of $\pi_1(\B\Sigma_g, z_i)$ shown in \eqref{eq:Pi-1-i-presentation}. Let $\{h'_t\}$ be a path in $\Diff_0(\B\Sigma_g,\area)$ which fixes the outside and a small neighborhood of $V_{i,j}$, rotates $V_{i,j}$ in the same direction, and rotates $W_{i,j}$ exactly once. Such $\{h'_t\}$ twists all $U_i$'s in the form of a generator $A_{j,i}\in\mathcal{S}\setminus\mathcal{S}'$.

We compose all the paths $\{h_t\}$'s and $\{h'_t\}$'s such that the time-one map of the composition flow $\{f_t\}\in \Diff_0(\B\Sigma_g,\area)$ twists all $U_i$'s in the form of pure braid $\b$. Set $f:=f_1$, then $f$ is identity on each $U_i$ and $\g(f;(x_1,\ldots,x_n))=\b$ for $x_i\in U_i$. If we denote by $U=U_1\cup\ldots\cup U_n$, then

\begin{equation}\label{eq:2-term}
\G_{n,g}(\phi)(f)=\int\limits_{x\in\X_n(U)}\varphi(\g(f;{x}))d{x}+
\int\limits_{x\notin\X_n(U)}\lim_{p\to\infty}\frac{\varphi(\g(f^p;{x}))}{p}\thinspace d{x}\thinspace .
\end{equation}

Let us start with the second term in the above equation. We denote by $\mathcal{A}:=\{\s_i\}_{i=1}^{n-1}$ to be the set of Artin generators for the Artin braid group $\B B_n$. Hence the group $\B B_n(\B\Sigma_g)$ is generated by the set $\mathbf{S}:=\mathcal{A}\cup(\mathcal{S}\setminus\mathcal{S}')$, see \cite{Bel}. For $x\in \X_n(\B\Sigma_g)$ denote by $cr(f^p;x)$ the length of the word in generators from $\mathbf{S}$, which represent the braid $\g(f^p;x)$ and is given by  $p$ concatenations of the flow $\{f_t\}$. Let
$cr(\b):=cr(f;z)$. It follows from the construction of $f$ that there exists a constant $C_n>0$, such that for each $x\in \X_n(\B\Sigma_g)$ we have
$$\lim_{p\to +\infty}\frac{cr(\g(f^p;x))}{p}\leq C_n\cdot cr(\b)\thinspace .$$
For each $\g\in\B B_n(\B\Sigma_g)$ denote by $l(\g)$ the word length of $\g$ with respect to the generating set $\mathbf{S}$. Since $\varphi$ is a homogenous quasi-morphism, we obtain
$$|\varphi(\g)|\leq \left(D_{\varphi}+\max_{s\in\mathbf{S}}|\varphi(s)|\right)\cdot l(\g)\thinspace .$$
It follows that for each $x\in \X_n(\B\Sigma_g)$ we have
\begin{eqnarray*}
\lim\limits_{p\to +\infty}\frac{|\varphi(\g(f^p;x))|}{p}&\leq&\left(D_{\varphi}+\max_{s\in\mathbf{S}}|\varphi(s)|\right)
\lim\limits_{p\to +\infty}\frac{|l(\g(f^p;x))|}{p}\\
&\leq& \left(D_{\varphi}+\max_{s\in\mathbf{S}}|\varphi(s)|\right)\lim\limits_{p\to +\infty}\frac{cr(\g(f^p;x))}{p}\\
&\leq& \left(D_{\varphi}+\max_{s\in\mathbf{S}}|\varphi(s)|\right)\cdot C_n\cdot cr(\b)\thinspace .
\end{eqnarray*}
It follows that
$$
\int\limits_{x\notin\X_n(U)}\lim_{p\to\infty}\frac{\varphi(\g(f^p;{x}))}{p}
\thinspace d{x}\xrightarrow[\area(U)\rightarrow\area(\B\Sigma_g)]{}0\thinspace .
$$

Now we turn to the first term in \eqref{eq:2-term}. There exists a constant $K>0$ such that in every neighborhood of $\area(\B\Sigma_g)$ there exists a number $\area(U)$, where $U=U_1\cup\ldots\cup U_n$ and $U_i$'s as before, such that
$$
\left|\thinspace\int\limits_{x\in\X_n(U)}\varphi(\g(f;{x}))d{x}\right|\geq K\thinspace .
$$
The proof of this fact is identical to the proof of \cite[Theorem 1.2]{I}. This concludes the proof of the theorem.
\qed

\begin{rem}\label{R:genuine-quasi-morphisms}
Let $\varphi$ be a genuine homogeneous quasi-morphism in $Q(\B B_n(\B\Sigma_g))$, i.e. $\varphi$ is not a homomorphism. Then the induced quasi-morphism $\G_{n,g}(\varphi)\in Q(\Diff_0(\B\Sigma_g,\area))$ is not a homomorphism. The proof of this fact is very similar to the proof of Theorem \ref{T:prop-GG-map}, and is omitted.
\end{rem}

Recall that for a group $\Gamma$ we denoted by $\widehat{Q}(\Gamma)$ the space of homogeneous quasi-morphisms on $\Gamma$ modulo the space of homomorphisms on $\Gamma$.

\begin{cor}\label{cor:inf-dim}
The vector space $\widehat{Q}(\Diff_0(\B\Sigma_g,\area))$ is infinite dimensional for $g>1$. Moreover, for $g>1$, the map
$$\G_{n,g}\colon \widehat{Q}(\B B_n(\B\Sigma_g))\to \widehat{Q}(\Diff_0(\B\Sigma_g,\area))\subset\widehat{Q}(\Ham(\B\Sigma_g))$$
is injective and hence the vector space $\widehat{Q}(\Ham(\B\Sigma_g))$ is also infinite dimensional.
\end{cor}

\begin{proof}
It is a well known fact that for each $g>1$ the braid group $\B B_n(\B\Sigma_g)$ is a subgroup of the mapping class group of the surface $\B\Sigma_g$ with $n$ punctures, see \cite{Bir1}. Since the group $\B B_n(\B\Sigma_g)$ is not virtually abelian, the space $\widehat{Q}(\B B_n(\B\Sigma_g))$ is infinite dimensional, see \cite[Theorem 12]{BF}. It follows from Theorem \ref{T:prop-GG-map} and Remark \ref{R:genuine-quasi-morphisms} that the space $\widehat{Q}(\Diff_0(\B\Sigma_g,\area))$ is also infinite dimensional.

It is known that, for $g>1$, the group $\Ham(\B\Sigma_g)$ is simple and it coincides with the kernel of Calabi homomorphism and with the commutator subgroup of $\Diff_0(\B\Sigma_g,\area)$, see \cite{Ba1}. It follows that there are no non-trivial quasi-morphisms in $\widehat{Q}(\Diff_0(\B\Sigma_g,\area))$ which are homomorphisms when restricted on $\Ham(\B\Sigma_g)$. Hence the map $\G_{n,g}$ is injective and its image in $\widehat{Q}(\Ham(\B\Sigma_g))$ is infinite dimensional.
\end{proof}

\begin{rem}
The above proof gives another proof to the fact that the spaces $\widehat{Q}(\Diff_0(\B\Sigma_g,\area))$ and $\widehat{Q}(\Ham(\B\Sigma_g))$ are infinite dimensional for $g>1$. Originally this fact was proved by Gambaudo-Ghys in \cite{GG}. Let $\B\Sigma_{g,k}$ be a compact orientable surface of genus $g$ with $k$ boundary components. Then, by the proof above, the spaces $\widehat{Q}(\Diff_0(\B\Sigma_{g,k},\area))$ are also infinite dimensional for $g>1$.
\end{rem}

\subsection{Braids traced by Morse autonomous flows}
\label{sec-braids-aut-flows}
Let $g,n>1$ and $h_t$ be an autonomous flow generated by a Morse function $H\colon\B\Sigma_g\to\B R$. We pick a point $x=(x_1,\ldots,x_n)\in\X_n(\B\Sigma_g)$ which satisfies the following conditions:

\begin{itemize}
\item
$x_i$ is a regular point for each $1\leq i\leq n$,
\item
connected components of $H^{-1}(H(x_i))$ and $H^{-1}(H(x_j))$, which are simple closed curves, are disjoint for each $1\leq i\neq j\leq n$.
\end{itemize}

Such a set of points in $\X_n(\B\Sigma_g)$ is denoted by $\mathrm{Reg}_H$. Note that the measure of $\X_n(\B\Sigma_g)\setminus\mathrm{Reg}_H$ is zero. If $x_i$ is a regular point of $H$, then the image of the path $\B R\to~\B\Sigma_g$, where $t\to h_t(x_i)$, is a simple closed curve. Let $c\colon[0,1]\to\B\Sigma_g$ be an injective path (on $(0,1)$), such that $c(0)=c(1)$ and its image is a simple closed curve defined above. Define
\begin{equation}\label{eq:gamma-reg}
\gamma_{x_i}(t):=
\begin{cases}
s_{z_ix_i}(3t) &\text{ for } t\in \left [0,\frac13\right ]\\
c_{3t-1}(x_i) &\text{ for } t\in \left [\frac13,\frac23\right ]\\
s_{x_iz_i}(3t-2) & \text{ for } t\in \left [\frac23,1\right ]
\end{cases}
\end{equation}
where the path $s$ was defined in Section \ref{SS:GG+results}.
For almost every $x\in\mathrm{Reg}_H$ and for each $1\leq i\leq n$ the path
$$\g_{h,i}(t):=(\widetilde{s}_{x_1}(t),\ldots,\widetilde{s}_{x_{i-1}}(t),\gamma_{x_i}(t),\widetilde{s}_{x_{i+1}}(t),\ldots,\widetilde{s}_{x_n}(t)),\quad t\in[0,1]$$
is in $\X_n(\B\Sigma_g)$, where
$$\widetilde{s}_{x_i}(t):=
\begin{cases}
s_{z_ix_i}(3t) &\text{ for } t\in \left [0,\frac13\right ]\\
x_i &\text{ for } t\in \left [\frac13,\frac23\right ]\\
s_{x_iz_i}(3t-2) & \text{ for } t\in \left [\frac23,1\right ]
\end{cases}.
$$
Let $x=(x_1,\ldots,x_n)\in\mathrm{Reg}_H$, and let $\b_{i,x}$ be a braid in $\B B_n(\B\Sigma_g)$ represented by the path $\g_{h,i}(t)$. Since the connected components of $H^{-1}(H(x_i))$ and $H^{-1}(H(x_j))$ are disjoint for $1\leq i\neq j\leq n$, the braids $\b_{i,x},\b_{j,x}$ commute for each $1\leq i,j\leq n$. For each $p\in\B N$ the braid $\g(h^p,x)$ can be written as a product
\begin{equation}\label{eq:braid-morse-aut}
\g(h_1^p,x)=\a'_{p,x}\circ\b_{1,x}^{k_{h_1,p,1}}\circ\ldots\circ
\b_{n,x}^{k_{h_1,p,n}}\circ\a''_{p,x}\thinspace,
\end{equation}
where $k_{h_1,p,i}$ is an integer which depends on $h_1$, $p$ and $x_i$, and the length of braids $\a'_{p,x}\thinspace, \a''_{p,x}$ is bounded by some constant $M(n)$ which depends only on $n$.

Denote by $\mathcal{MCG}_g^n$ the mapping class group of a surface $\B\Sigma_g$ with $n$ punctures. There is a following short exact sequence due to Birman \cite{Bir1}
\begin{equation}\label{eq:Birman-SES}
1\to \B B_n(\B\Sigma_g)\to\mathcal{MCG}_g^n\to\mathcal{MCG}_g\to 1,
\end{equation}
where $\mathcal{MCG}_g$ is the mapping class group of a surface $\B\Sigma_g$. Hence the surface braid group
$\B B_n(\B\Sigma_g)$ may be viewed as an infinite normal subgroup of $\mathcal{MCG}_g^n$ for $g>1$. Note that if $g=1, n>1$ then $\B B_n(\B\Sigma_1)$ is no longer a subgroup of $\mathcal{MCG}_1^n$, see \cite{Bir1}. Hence the case of the torus is different and will not be treated in this paper.

\begin{prop}\label{P:finite-conj-classes}
Let $g>1$. Then there exists a finite set $S_{g,n}$ of elements in $\mathcal{MCG}_g^n$, such that for a diffeomorphism $h_1\in\Ham(\B\Sigma_g)$, which is generated by a Morse function $H\colon\B\Sigma_g\to\B R$, and for each $x\in\mathrm{Reg}_H$ the braid $\b_{i,x}$ is conjugated in $\mathcal{MCG}_g^n$ to some element in $S_{g,n}$.
\end{prop}
\begin{proof} Let $x\in\mathrm{Reg}_H$.

\textbf{Case 1.} Suppose that $\b_{i,x}$ is such that the image of the path $t\to h_t(x_i)$ is homotopically trivial in $\B\Sigma_g$. Hence there exists an embedded disc in $\B\Sigma_g$ whose boundary is the image of this curve. There exists some $j\in\{0,\ldots,n-1\}$, such that there are exactly $j$ points from the set $\{x_l\}_{l=1, l\neq i}^n$, that lie in this disc. It follows from the definition of $\b_{i,x}$ that it is conjugate in $\B B_n(\B\Sigma_g)$ (and hence in $\mathcal{MCG}_g^n$) to the braid $\eta_{j+1,n}\in \B B_n<\B B_n(\B\Sigma_g)$ shown in Figure \ref{fig:braids-eta-j-n}.

\begin{figure}[htb]
\centerline{\includegraphics[height=1.6in]{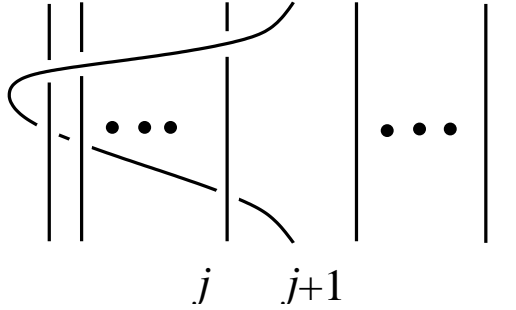}}
\caption{\label{fig:braids-eta-j-n} Braid $\eta_{j+1,n}$}
\end{figure}

\textbf{Case 2.} Suppose that $\b_{i,x}$ is such that the image of the path $t\to h_t(x_i)$ is homotopically non-trivial in $\B\Sigma_g$. Denote by $\B\Sigma_{g,n-1}$ a surface $\B\Sigma_g$ with $n-1$ punctures. We say that simple closed curves $\delta$ and $\delta'$ are equivalent $\delta\cong\delta'$ if there exists a homeomorphism
$$f\colon\B\Sigma_{g,n-1}\to\B\Sigma_{g,n-1}$$
such that $f(\delta)=\delta'$. It follows from classification of surfaces that the set of equivalence classes $\mathcal{R}_g$ is finite.

Denote by $\delta_h$ a simple closed curve which is an image of the path $t\to h_t(x_i)$. Since $\B\Sigma_g$ and $\delta_h$ are oriented, the curve $\delta_h$ splits in $\B\Sigma_g\setminus\{x_1,\ldots,x_n\}$ into two simple closed curves $\delta_{h,-}$ and $\delta_{h,+}$ which are homotopic in $\B\Sigma_g\setminus\{x_1,\ldots,x_{i-1},x_{i+1},\ldots,x_n\}$, see Figure \ref{fig:path-split}.
\begin{figure}[htb]
\centerline{\includegraphics[height=1.5in]{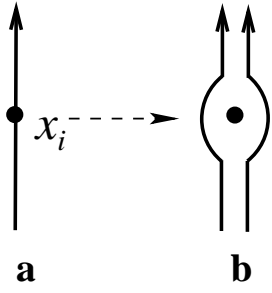}}
\caption{\label{fig:path-split} Part of the curve $\delta_h$ is shown in Figure \textbf{a}. Its splitting into curves $\delta_{h,-}$ and $\delta_{h,+}$ is shown in Figure \textbf{b}. The left curve in Figure \textbf{b} is $\delta_{h,-}$ and the right curve is $\delta_{h,+}$}
\end{figure}

The image of the braid $\b_{i,x}$ in $\mathcal{MCG}_g^n$ under the Birman embedding of $\B B_n(\B\Sigma_g)$ into $\mathcal{MCG}_g^n$ is conjugated to $t_{\delta_{h,+}}\circ t^{-1}_{\delta_{h,-}}$, where $t_{\delta_{h,+}}$ and $t_{\delta_{h,-}}$ are Dehn twists in $\B\Sigma_g\setminus\{x_1,\ldots,x_n\}$ about $\delta_{h,+}$ and $\delta_{h,-}$ respectively. Note that if $\delta_h\cong\delta'$ then there exists a homeomorphism $f\colon\B\Sigma_{g,n-1}\to\B\Sigma_{g,n-1}$ such that $f(\delta_h)=\delta'$, hence $f(\delta_{h,+})=\delta'_+$ and $f(\delta_{h,-})=\delta'_-$. We have
$$
t_{\delta'_+}=f\circ t_{\delta_{h,+}}\circ f^{-1}\qquad t_{\delta'_-}=f\circ t_{\delta_{h,-}}\circ f^{-1}.
$$
This yields
$$f\circ(t_{\delta_{h,+}}\circ t^{-1}_{\delta_{h,-}})\circ f^{-1}=t_{\delta'_+}\circ t^{-1}_{\delta'_-}.$$
In other words, the braid $\b_{i,x}$ is conjugated in $\mathcal{MCG}_g^n$ to some $t_{\delta_+}\circ t^{-1}_{\delta_-}$, where $\delta$ is some representative of an equivalence class in $\mathcal{R}_g$. Let $\{\delta_i\}_{i=1}^{\#\mathcal{R}_g}$ be a set of simple closed curves in $\B\Sigma_g$, such that each equivalence class in $\mathcal{R}_g$ is represented by some $\delta_i$. Let
$$S_{g,n}:=\left\{\eta_{1,n}^{\pm 1},\ldots,\eta_{n,n}^{\pm 1},
\left(t_{\delta_{1,+}}\circ t^{-1}_{\delta_{1,-}}\right)^{\pm 1},\ldots,
\left(t_{\delta_{\#\mathcal{R}_g,+}}\circ t^{-1}_{\delta_{\#\mathcal{R}_g,-}}\right)^{\pm 1}\right\}.$$
It follows that $\b_{i,x}$ is conjugated to some element in $S_{g,n}$.
Note that the set $S_{g,n}$ is independent of $H$, $h_1$ and $x\in\mathrm{Reg}_H$ and the proof follows.
\end{proof}

\subsection{Braids traced by general autonomous flows}
\label{sec-braids-gen-aut-flows}
Let $g,n>1$ and $h_t$ be an autonomous flow generated by a function $H\colon\B\Sigma_g\to\B R$. Let $e\in\B R$ such that $H^{-1}(e)$ is a critical level set of $H$. Since $h_t$ is an autonomous flow, the set $H^{-1}(e)$ decomposes as a union of curves
\begin{equation}\label{eq:decomposition}
H^{-1}(e)=\bigcup\limits_{y\in H^{-1}(e)}h_t(y),
\end{equation}
where the image of each curve $h_t(y):\B R\to H^{-1}(e)$ is a point or a simple curve. We pick a point $x=(x_1,\ldots,x_n)\in\X_n(\B\Sigma_g)$ which satisfies the following conditions:

\begin{itemize}
\item
$x_i$ is a regular or critical point for each $1\leq i\leq n$.
\item
For each pair $(x_i, x_j)$, where at least one of the points belongs to the regular level set of $H$, we require the connected components of the sets $H^{-1}(H(x_i))$ and $H^{-1}(H(x_j))$ are disjoint for each $1\leq i\neq j\leq n$.
\item
For each pair $(x_i, x_j)$, where both $x_i,x_j$ belong to the same critical level set $H^{-1}(e)$, we require that they lie on disjoint curves in the decomposition of $H^{-1}(e)$ as a union of curves, see ~\eqref{eq:decomposition}.
\end{itemize}

Such a set of points in $\X_n(\B\Sigma_g)$ is denoted by $\mathrm{Gen}_H$. Note that the measure of $\X_n(\B\Sigma_g)\setminus\mathrm{Gen}_H$ is zero, and $\mathrm{Reg}_H\subset\mathrm{Gen}_H$.

Let $x\in\X_n(\B\Sigma_g)$. If $x_i$ belongs to the regular level set of $H$, then $\gamma_{x_i}(t)$ is defined in ~\eqref{eq:gamma-reg}. Suppose that $x_i$ belongs to the critical level set of $H$. The image of the path $\B R\to~\B\Sigma_g$, where $t\to h_t(x_i)$, is a simple curve, because $h_t(x_i)$ is an integral curve of a time-independent vector field. For each $p\in\B N$ set
$$
\gamma_{x_i,p}(t):=
\begin{cases}
s_{z_ix_i}(3t) &\text{ for } t\in \left [0,\frac13\right ]\\
h_{p(3t-1)}(x_i) &\text{ for } t\in \left [\frac13,\frac23\right ]\\
s_{h_p(x_i)z_i}(3t-2) & \text{ for } t\in \left [\frac23,1\right ],
\end{cases}
$$
where the path $s$ was defined in Section \ref{SS:GG+results}.

For almost every $x\in\mathrm{Gen}_H$, for each $1\leq i\leq n$ and $p\in\B N$, the path
$$\g_{h,i,p}(t):=(\widetilde{s}_{x_1}(t),\ldots,\widetilde{s}_{x_{i-1}}(t),\gamma_{x_i,p}(t),\widetilde{s}_{x_{i+1}}(t),\ldots,\widetilde{s}_{x_n}(t)),\quad t\in[0,1]$$ is in $\X_n(\B\Sigma_g)$, where
$$\widetilde{s}_{x_i}(t):=
\begin{cases}
s_{z_ix_i}(3t) &\text{ for } t\in \left [0,\frac13\right ]\\
x_i &\text{ for } t\in \left [\frac13,\frac23\right ]\\
s_{x_iz_i}(3t-2) & \text{ for } t\in \left [\frac23,1\right ]
\end{cases}.
$$
Let $\b_{i,x,p}$ be a braid in $\B B_n(\B\Sigma_g)$ represented by the path $\g_{h,i,p}(t)$. Let $x\in\mathrm{Gen}_H$. Since the curves $h_{p(3t-1)}(x_i)$ and $h_{p(3t-1)}(x_j)$ are disjoint for $1\leq i\neq j\leq n$, the braids $\b_{i,x,p},\b_{j,x,p}$ commute for each $1\leq i,j\leq n$. For each $p\in\B N$ the braid $\g(h^p,x)$ can be written as a product
\begin{equation}\label{eq:braid-gen-aut}
\g(h_1^p,x)=\a'_{p,x}\circ\b_{1,x,p}\circ\ldots\circ\b_{n,x,p}\circ\a''_{p,x}\thinspace,
\end{equation}
where the length of braids $\a'_{p,x}\thinspace, \a''_{p,x}$ is bounded by some constant $M(n)$ which depends only on $n$. Moreover, if $x_i$ belongs to a critical level set $H^{-1}(e)$, then:
\begin{figure}[htb]
\centerline{\includegraphics[height=1.2in]{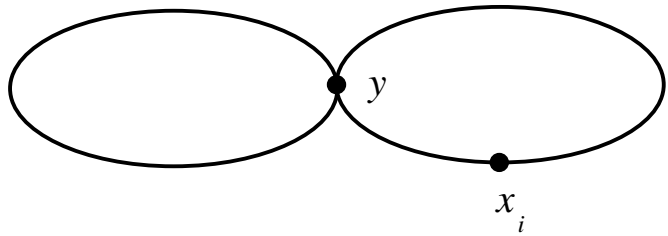}}
\caption{\label{fig:fig8-level-set} Suppose that for some $e\in\B R$, the set $H^{-1}(e)$ is a figure eight curve. Hence $y$ is a critical point, and suppose that $x_i$ is a regular point. The point $y$ can not be in the image of the curve $\B R\to h_t(x_i)$ and the curve $h_t(x_i)$ does not reverse the orientation, because $h_t$ is an autonomous flow. It follows that the curve $h_t(x_i)$ has a finite length.}
\end{figure}

\begin{itemize}
\item
if $x_i$ is a critical point of $H$, then the braid $\b_{i,x,p}$ is the identity for every $p$, since $\forall t\in\B R$ the flow $h_t(x_i)=x_i$.

\item
if $x_i$ belongs to a critical set $H^{-1}(e)$, but $x_i$ is not a critical point, then $x_i$ belongs to a simple curve in the decomposition of $H^{-1}(e)$ into union of curves. In this case, the length of the braid $\b_{i,x,p}$ is bounded, uniformly for all $p$, by some constant $M'(n)$ which depends only on $n$. This follows from the fact that the curve $h_t(x_i)$ has a finite length. For example, see Figure \ref{fig:fig8-level-set}.
\end{itemize}
Also note that if $x_i$ belongs to regular level set of $H$, then the braid $\b_{i,x,p}=\b_{i,x}^{k_{h_1,p,i}}$,
where $\b_{i,x}^{k_{h_1,p,i}}$ is the braid defined in ~\eqref{eq:braid-morse-aut}.

\begin{prop}\label{P:gen-finite-conj-classes}
Let $g>1$. Then for $h_1\in\Ham(\B\Sigma_g)$, which is generated by a function $H\colon\B\Sigma_g\to\B R$ and for $x=(x_1,\ldots,x_n)\in\mathrm{Gen}_H$ there exists $k_p\in\B Z$, such that for each $x_i$ the braid $\b_{i,x,p}$, defined in ~\eqref{eq:braid-gen-aut}, is conjugated to $s^{k_p}$, where $s$ is some element in the set $S_{g,n}$ which was defined in Proposition \ref{P:finite-conj-classes}.
\end{prop}
\begin{proof}
The proof is identical to the proof of Proposition \ref{P:finite-conj-classes}.
\end{proof}

\subsection{Proof of Theorem \ref{T:main}}
\label{sec-proof-main}
Recall that we view $\B B_n(\B\Sigma_g)$ as a normal subgroup of $\mathcal{MCG}_g^n$ and $\B B_n<\B B_n(\B\Sigma_g)$. For each $g,n>1$ denote by $Q_{\mathcal{MCG}_g^n}(\B B_n, S_{g,n})$ the space of homogeneous quasi-morphisms on $\B B_n(\B\Sigma_g)$ so that:
\begin{itemize}
\item
For each $\varphi\in Q_{\mathcal{MCG}_g^n}(\B B_n, S_{g,n})$ there exists $\widehat{\varphi}\in Q(\mathcal{MCG}_g^n)$ such that $\widehat{\varphi}|_{\B B_n}=\varphi|_{\B B_n}$, and $\widehat{\varphi}$ vanishes on the finite set $S_{g,n}$.

\item there are infinitely many linearly independent homogeneous quasi-morphisms in $Q_{\mathcal{MCG}_g^n}(\B B_n, S_{g,n})$ that remain linearly independent when restricted to $\B B_n$.
\end{itemize}

In \cite{BBF} Bestvina-Bromberg-Fujiwara classified all elements $\g\in\mathcal{MCG}_g^n$ for which there is a quasi-morphism on $\mathcal{MCG}_g^n$ which is unbounded on the powers of $\g$, in terms of the Nielsen-Thurston decomposition. As a Corollary they proved the following

\begin{prop}[Proposition 4.9 in \cite{BBF}]\label{P:BBF}
Let $\B\Sigma'\subset\B\Sigma_{g,n}$ be a subsurface that supports pseudo-Anosov
homeomorphisms. Then the restriction map
$$\widehat{Q}(\mathcal{MCG}_g^n)\to \widehat{Q}(\mathcal{MCG}(\B\Sigma'))$$
has infinite dimensional image. Here $\mathcal{MCG}(\B\Sigma')$ is a mapping class group of a subsurface $\B\Sigma'$.
\end{prop}

Let $n>2$. We identify $\B B_n<\mathcal{MCG}_g^n$ with a mapping class group of an $n$-punctured disc. Since there are infinitely many pseudo-Anosov braids in $\B B_n$, it follows from Proposition \ref{P:BBF} that there exist infinitely many linearly independent homogeneous quasi-morphisms on $\mathcal{MCG}_g^n$ that remain linearly independent when restricted to $\B B_n$. Combining this with the fact that the set $S_{g,n}$ is finite we have the following important

\begin{cor}\label{cor:imp-inf-dim-cor}
Let $g>1$, $n>2$. Then $\dim(Q_{\mathcal{MCG}_g^n}(\B B_n, S_{g,n}))=\infty$.
\end{cor}
Note that there are no non-trivial homomorphisms in $Q_{\mathcal{MCG}_g^n}(\B B_n, S_{g,n})$.

\begin{thm}\label{T:inf-diam}
For each $g>1$ the metric space $(\Ham(\B\Sigma_g), \B d_{\Aut})$ has an infinite diameter.
\end{thm}
\begin{proof}
Recall that $\OP{Aut}\subset \Ham(\B\Sigma_g)$ is the set of all autonomous diffeomorphisms of $\B\Sigma_g$. Denote by
$Q(\Ham(\B\Sigma_g), \OP{Aut})$ the space of homogeneous quasi-morphisms on $\Ham(\B\Sigma_g)$ that vanish on the set $\OP{Aut}$. Since the set $\OP{Aut}$ generates $\Ham(\B\Sigma_g)$, there are no non-trivial homomorphisms in $Q(\Ham(\B\Sigma_g), \OP{Aut})$.

\begin{lem}\label{L:Aut-zero}
Let $g>1$ and $n>2$. Then for each $\varphi\in Q_{\mathcal{MCG}_g^n}(\B B_n,S_{g,n})$ and
an autonomous diffeomorphism $h\in\Ham(\B\Sigma_g)$ we have $$\G_{n,g}(\varphi)(h):=\overline{\Phi}_n(h)=0.$$
\end{lem}
\begin{proof}
There are two possible proofs of this lemma.

\textbf{First proof.} Note that the set of Morse functions on $\Sigma_g$ is $C^1$-dense in the set of smooth functions on $\Sigma_g$, see \cite{Mil}. Denote the set of Morse functions from $\B\Sigma_g$ to $\B R$ by $\cH_g$. Arguing similarly as in the proof of \cite[Theorem 3.4]{BK2} one can prove the following

\begin{thm}\label{T:Morse-Ham}
Let $H\colon\B\Sigma_g\to\B R$ be a Hamiltonian function and $\{H_k\}_{k=1}^\infty$ be a family of functions so that
each $H_k\in\cH_g$ and $H_k\xrightarrow[k\rightarrow\infty]{}~H$ in
$C^1$-topology. Let $h_1$ and $h_{1,k}$ be the time-one maps of the Hamiltonian flows
generated by $H$ and $H_k$ respectively. Then
$$
\lim\limits_{k\to\infty}\overline{\Phi}_n(h_{1,k})=\overline{\Phi}_n(h_1),
$$
where $\overline{\Phi}_n\colon\Ham(\B\Sigma_g)\to\B R$ is a quasi-morphism defined in ~\eqref{eq:GG-ext}.
\end{thm}

Then, by Theorem \ref{T:Morse-Ham}, it is enough to prove this lemma for $h_1$ generated by some Morse function $H$. For almost every
$x\in\X_n(\B \Sigma_g)$ we have
\begin{equation*}
\g(h_1^p,x)=\a'_{p,x}\circ\b_{1,x}^{k_{h_1,p,1}}\circ\ldots\circ
\b_{n,x}^{k_{h_1,p,n}}\circ\a''_{p,x}\thinspace,
\end{equation*}
see \eqref{eq:braid-morse-aut}. Since all $\b_{i,x}$'s commute we have
$$|\varphi(\g(h_1^p,x))|\leq 2 D_\varphi+|\varphi(\a'_{p,x})|+|\varphi(\a''_{p,x})|+\sum_{i=1}^n |k_{h_1,p,i}|\cdot|\varphi(\b_{i,x})|\thinspace .$$
By Proposition \ref{P:finite-conj-classes} each $\b_{i,x}$ is conjugate to some element in $S_{g,n}$, hence $\varphi(\b_{i,x})=0$. Since the length of braids $\a'_{p,x}$ and $\a''_{p,x}$ is bounded by $M(n)$, there exists a constant $M_1(n)$ such that $|\varphi(\a'_{p,x})|<M_1(n)$ and $|\varphi(\a''_{p,x})|<M_1(n)$.
It follows that
\begin{equation}\label{eq:morse-aut-ineq}
|\varphi(\g(h_1^p,x))|\leq 2 (D_\varphi+M_1(n)).
\end{equation}
This yields
$$|\overline{\Phi}_n(h_1)|\leq\lim_{p\to\infty}\int\limits_{\X_n(\B\Sigma_g)}\frac{1}{p}|\varphi(\g(h_1^p,x))|dx=0,$$
where the equality follows immediately from \eqref{eq:morse-aut-ineq}, and the proof follows.

\textbf{Second proof.} This proof does not require any continuity properties of quasi-morphisms. Recall that for $x=(x_1,\ldots,x_n)\in\mathrm{Gen}_H$ we have the following decomposition
$$
\g(h_1^p,x)=\a'_{p,x}\circ\b_{1,x,p}\circ\ldots\circ
\b_{n,x,p}\circ\a''_{p,x}\thinspace,
$$
see ~\eqref{eq:braid-gen-aut}. Since all $\b_{i,x,p}$'s commute we have
$$|\varphi(\g(h_1^p,x))|\leq 2 D_\varphi+|\varphi(\a'_{p,x})|+|\varphi(\a''_{p,x})|+\sum_{i=1}^n |\varphi(\b_{i,x,p})|\thinspace .$$

We showed that if $x_i$ belongs to a critical level set of $H$, then the length of $\b_{i,x,p}$ is uniformly bounded by some $M'(n)$. If $x_i$ belongs to a regular level set, then by Proposition \ref{P:gen-finite-conj-classes}, the braid $\b_{i,x,p}$ is conjugate to some $s^{k_p}$, where $s\in S_{g,n}$. Since $\varphi\in Q_{\mathcal{MCG}_g^n}(\B B_n,S_{g,n})$, then in both cases
$$\lim\limits_{p\to\infty}\frac{\varphi(\b_{i,x,p})}{p}=0.$$
Recall that $|\varphi(\a'_{p,x})|<M_1(n)$ and $|\varphi(\a''_{p,x})|<M_1(n)$.
It follows that
$$
\frac{|\varphi(\g(h_1^p,x))|}{p}\leq \frac{2 (D_\varphi+M_1(n))+\sum_{i=1}^n |\varphi(\b_{i,x,p})|}{p}\thinspace.
$$
This yields
$$|\overline{\Phi}_n(h_1)|\leq\lim_{p\to\infty}\int\limits_{\X_n(\B\Sigma_g)}\frac{1}{p}|\varphi(\g(h_1^p,x))|d\overline{x}=0,$$
and the proof follows.
\end{proof}

It follows from Corollary \ref{cor:inf-dim} and from Lemma \ref{L:Aut-zero} that
$$\G_{3,g}\colon Q_{\mathcal{MCG}_g^3}(\B B_3,S_{g,3})\hookrightarrow Q(\Ham(\B\Sigma_g), \OP{Aut})\thinspace .$$
By Corollary \ref{cor:imp-inf-dim-cor} we have
$$\dim(Q_{\mathcal{MCG}_g^3}(\B B_3,S_{g,3}))=\infty,$$
hence
\begin{equation}\label{eq:inf-dim}
\dim(Q(\Ham(\B\Sigma_g), \OP{Aut}))=\infty\thinspace .
\end{equation}

\begin{rem}\label{R:unbounded}
The existence of a genuine (that is not a homomorphisms)
homogeneous quasi-morph\-ism $\psi\colon \Ham(\B\Sigma_g)\to \B R$ that is trivial on the set
$\OP{Aut}\subset \Ham(\B\Sigma_g)$ implies that the autonomous norm is unbounded. Indeed,
for every $f\in \Ham(\B\Sigma_g)$
we have that
$$|\psi(f)|=|\psi(h_1\circ\ldots\circ h_m)|\leq mD_\psi$$
and hence for every
natural number $n$ we get $\|f^n\|_{\OP{Aut}}\geq \frac{|\psi(f)|}{D_\psi}n>0$,
provided $\psi(f)\neq 0$.
\end{rem}
Remark \ref{R:unbounded} concludes the proof of the theorem.
\end{proof}

\begin{rem}\label{R:proof-without-braids}
In order to show that for each $g>1$ the metric space $(\Ham(\B\Sigma_g), \B d_{\Aut})$ has an infinite diameter it is enough to consider the case when $n=1$, i.e. the space $\widehat{Q}(\pi_1(\B\Sigma_g))$, and the generalized Gambaudo-Ghys construction is not needed. For more details see \cite{B-autonomous}. However, we will see that the proof of Theorem \ref{T:main} requires the generalized Gambaudo-Ghys construction.
\end{rem}

Let us proceed with the proof of Theorem \ref{T:main}. Choose $r<\frac{1}{m}$. There exists a family $\{\varphi_i\}_{i=1}^m$ of linearly independent homogeneous quasi-morphisms in $Q_{\mathcal{MCG}_g^3}(\B B_3,S_{g,3})$, and a family of diffeomorphisms $\{h_j\}_{j=1}^m$ in $\Ham(\B\Sigma_g)$, where the support of each $h_j$ is contained in some disc of area $r$, such that
\begin{equation}\label{eq:delta-ij}
\overline{\Phi}_i(h_j):=\G_{3,g}(\varphi_i)(h_j)=\delta_{ij},
\end{equation}
where $\delta_{ij}$ is the Kronecker delta. This follows from Theorem \ref{T:prop-GG-map} combined with \eqref{eq:inf-dim} and with Lemma 3.10 and Lemma 3.11 in \cite{BK2}.
Since $r<\frac{1}{m}\thinspace$, there exists a family of diffeomorphisms
$\{f_j\}_{j=1}^m$ in $\B \Ham(\B \Sigma_g)$, such that $f_j\circ h_j\circ f_j^{-1}$ and
$f_i\circ h_i\circ f_i^{-1}$ have disjoint supports for all different $i$ and
$j$ between $1$ and $m$. Denote by $g_i:=f_i\circ h_i\circ f_i^{-1}$. Since all $g_i$'s have disjoint supports, they generate a free abelian group of rank $m$. Let
$$
T\colon\B Z^m\to\Ham(\B\Sigma_g)
$$
where $T(d_1,\ldots,d_m)=g_1^{d_1}\circ\ldots\circ g_m^{d_m}$.
It is obvious that $T$ is a monomorphism. In
order to complete the proof of the theorem it is left to show that $T$ is a
bi-Lipschitz map, i.e. we are going to show that there exists a constant
$A\geq 1$ such that
$$
A^{-1}\sum_{i=1}^m |d_i|\leq\|g_1^{d_1}\circ\ldots\circ g_m^{d_m}\|_{\OP{Aut}}
\leq A\sum_{i=1}^m |d_i|\thinspace.
$$

We have the following equalities
$$
\overline{\Phi}_i(g_j)
=\overline{\Phi}_i(f_j\circ h_j\circ f_j^{-1})
=\overline{\Phi}_i(h_j)
=\delta_{ij}\thinspace.
$$
The second equality follows from the fact that every homogeneous quasi-morphism
is invariant under conjugation, and the third equality is \eqref{eq:delta-ij}.
Since all $g_i$
commute and $\overline{\Phi}_i(g_j)=\delta_{ij}$, we obtain
$$
\|g_1^{d_1}\circ\ldots\circ g_m^{d_m}\|_{\OP{Aut}}
\geq \frac{|\overline{\Phi}_i(g_1^{d_1}\circ\ldots\circ
g_m^{d_m})|}{D_{\overline{\Phi}_i}}
=\frac{|d_i|}{D_{\overline{\Phi}_i}}\thinspace,
$$
where
$D_{\overline{\Phi}_i}$ is the defect of the quasi-morphism
$\overline{\Phi}_i$. The defect $D_{\overline{\Phi}_i}\neq~0$ because
each $\overline{\Phi}_i\in Q(\Ham(\B\Sigma_g), \OP{Aut})$.
Let $\mathfrak{D}_m:=\max\limits_i D_{\overline{\Phi}_i}$. We obtain the following inequality
\begin{equation}\label{eq:L-norms}
\|g_1^{d_1}\circ\ldots\circ g_m^{d_m}\|_{\OP{Aut}}
\geq(m\cdot\mathfrak{D}_m)^{-1}\sum_{i=1}^m |d_i|\thinspace .
\end{equation}
Denote by $\mathfrak{M}_m:=\max\limits_i\|g_i\|_{\OP{Aut}}$. Now we
have the following inequality
\begin{equation}\label{eq:easy-ineq}
\|g_1^{d_1}\circ\ldots\circ g_m^{d_m}\|_{\OP{Aut}}
\leq \sum_{i=1}^m |d_i|\cdot\|f_i\|_{\OP{Aut}}
\leq \mathfrak{M}_m\cdot \sum_{i=1}^m |d_i|\thinspace .
\end{equation}
Inequalities \eqref{eq:L-norms} and \eqref{eq:easy-ineq} conclude the proof of the theorem.
\qed

\begin{rem}
Recall that $\B\Sigma_{g,k}$ is a compact orientable surface of genus $g$ with $k$ boundary components. Similarly to the case of closed surfaces the group $\Ham(\B\Sigma_{g,k})$ admits a bi-invariant autonomous metric. The proof presented above, shows that for every $g>1$ Theorem \ref{T:main} holds for a group $\Ham(\B\Sigma_{g,k})$ as well.
\end{rem}

\subsection{Proof of Theorem \ref{T:frag-metric}}
\label{sec-proof-frag}

At the beginning let us recall a construction, due to L. Polterovich, of quasi-morphisms on $\Ham(\B M)$.

Let $z\in \B M$. Suppose that the group $\pi_1(\B M,z)$ admits a \emph{non-trivial} homogeneous quasi-morphism $\phi\colon\pi_1(\B M,z)\to\B R$. For each $x\in\B M$ let us choose an arbitrary geodesic path from $x$ to $z$. In \cite{P} Polterovich constructed the induced \emph{non-trivial} homogeneous quasi-morphism $\widetilde{\Phi}$ on $\Ham(\B M)$ as follows:\\
For each $x\in\B M$ and an isotopy $\{h_t\}_{t\in[0,1]}$ between $\Id$ and $h$, let $g_x$ be a closed loop in $\B M$ which is a concatenation of a geodesic path from $z$ to $x$, the path $h_t(x)$ and a described above geodesic path from $h(x)$ to~$z$. Denote by $[h_x]$ the corresponding element in $\pi_1(\B M,z)$ and set
$$\Psi(h):=\int\limits_{\B M} \phi([h_x])\o^k\qquad\qquad
\overline{\Psi}(h):=\lim\limits_{p\to\infty}\frac{1}{p}\int\limits_{\B M} \phi([(h^p)_x])\o^k,$$
where $\dim(\B M)=2k$. The maps $\Psi$ and $\overline{\Psi}$ are well-defined quasi-morphisms because every diffeomorphism in $\Ham(\B M)$ is volume-preser\-ving. They do not depend on a choice of the path $\{h_t\}$ because the evaluation map $\ev\colon\Ham(\B M)\to \B M$, where $\ev(h)=h(z)$, induces a trivial map on $\pi_1(\Ham(\B M),\Id)$, see \cite{McDuff}. In addition, the quasi-morphism $\overline{\Psi}$ neither depends on the choice of a family of geodesic paths, nor on the choice of a base point $z$. For more details see \cite{P}. Note that if $\B M=\B\Sigma_g$ and $g>1$, then this is exactly the generalized Gambaudo-Ghys construction for $n=1$, i.e. $\B B_1(\B\Sigma_g)=\pi_1(\B\Sigma_g)$ and $\overline{\Psi}=\overline{\Phi}_1$. An immediate corollary from this construction is
\begin{cor}\label{cor:Polt-qm}
Let $\overline{\Psi}$ be a homogeneous quasi-morphism on $\Ham(\B M)$ given by Polterovich construction, then $\overline{\Psi}(h)=0$ for any $h\in\Ham(\B M)$ which is supported in a ball.
\end{cor}

We start proving the case \textbf{1}, where $\B M=\B\Sigma_g$ and $g>1$. There exist $2g-2$ different points $z_1,\ldots,z_{2g-2}$ and $2g-2$ disjoint embeddings $\emb_i$ of figure eight into $\B\Sigma_g$ such that:
\begin{itemize}
\item
an image of a singular point of figure eight under $\emb_i$ equals to some $z_i$,
\item
each $\emb_i$ induces an injective map $(\emb_i)_*\colon\B F_2\hookrightarrow\pi_1(\B \Sigma_g, z_i)$,
\item
the image of the induced map $Q(\pi_1(\B \Sigma_g, z_i))\to Q(\B F_2)$ is infinite dimensional for each $i$.
\end{itemize}
These embeddings are shown in Figure \ref{fig:fig8-emb}.
\begin{figure}[htb]
\centerline{\includegraphics[height=1.7in]{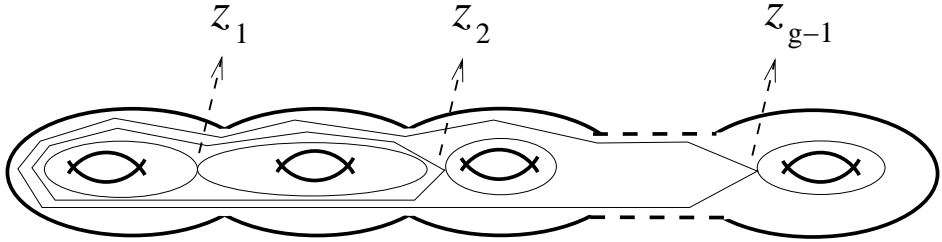}}
\caption{\label{fig:fig8-emb} Images of embeddings $\emb_1,\ldots\emb_{g-1}$ of figure eight are shown in this picture. The image of $\emb_{i+g-1}$ is a reflection of the image of $\emb_i$ over the horizontal plane for each $1\leq i\leq g-1$.}
\end{figure}

Let $a,b$ denote generators of $\B F_2$. Since every $\pi_1(\B \Sigma_g, z_i)$ is conjugate to every
$\pi_1(\B \Sigma_g, z_j)$, then by \cite[Lemma 3.10]{BK2} there exists a family of quasi-morphisms
$\{\phi_i\}_{i=1}^{2g-2}\in \widehat{Q}(\pi_1(\B \Sigma_g, z_i))$ and $[\g_i]\in\pi_1(\B\Sigma_g,z_i)$ such that:
\begin{itemize}
\item
each $[\g_i]$ is an image of some word in generators $a,b$ under the homomorphism $(\emb_i)_*\colon\B F_2\hookrightarrow\pi_1(\B \Sigma_g, z_i)$,
\item
we have $\phi_i([\g_j])=\delta_{ij}$ for $1\leq i,j\leq 2g-2$, where $\delta_{ij}$ is a Kronecker delta,
\item
and $\phi_i((\emb_i)_*(a))=\phi_i((\emb_i)_*(b))=0$ for $1\leq i\leq 2g-2$.
\end{itemize}
Let $\a_i$ and $\b_i$ be simple closed oriented curves in $\B\Sigma_g$ which represent $(\emb_i)_*(a)$ and $(\emb_i)_*(b)$ respectively. Denote by $w([\a_i],[\b_i])$ a reduced word in $[\a_i]$ and $[\b_i]$ which represents $[\g_i]$. Let $\varepsilon>0$ such that any $(2g-2)\times(2g-2)$ matrix $A$, where $1-\varepsilon<a_{ii}<1+\varepsilon$ and $|a_{ij}|<\varepsilon$ ($i\neq j$) is non-singular. In what follows we are going to construct, for each $1\leq i\leq 2g-2$, autonomous Hamiltonian flows $h_{t,i},g_{t,i}\in\Ham(\B\Sigma_g)$ such that

\begin{itemize}
\item
for each $i\neq j$ the support of the flow $h_{t,i}$ is disjoint from the supports of flows $h_{t,j}$ and $g_{t,j}$, and the support of the flow $g_{t,i}$ is disjoint from the supports of flows $h_{t,j}$ and $g_{t,j}$.
\item
We have $|\overline{\Psi}_i(f_i)-1|<\varepsilon$, where $\overline{\Psi}_i$ is a homogeneous quasi-morphism on $\Ham(\B\Sigma_g)$ induced by a quasi-morphism $\phi_i$, and $f_i$ is a Hamiltonian diffeomorphism $w(h_{1,i},g_{1,i})$, i.e. it is a word $w$ in diffeomorphisms $h_{1,i}$ and $g_{1,i}$.
\item We have $|\overline{\Psi}_i(f_j)|<\varepsilon$ for $i\neq j$.
\end{itemize}

\subsubsection{Construction of the flows $h_{t,i},g_{t,i}\in\Ham(\B\Sigma_g)$.} It follows from the tubular neighborhood theorem that for each $\a_i$ and $\b_i$ there exists $r_{\a},r_{\b}>0$, and area-preserving embeddings
$$
\OP{Emb}_{\a_i}:(0,r_{\a})\times \B S^1\hookrightarrow \B \Sigma_g \qquad \OP{Emb}_{\b_i}:(0,r_{\b})\times \B S^1\hookrightarrow \B \Sigma_g
$$
such that $\OP{Emb}_{\a_i}|_{\left\{\frac{r_{\a}}{2}\right\}\times \B S^1}$ is the curve $\a_i$, $\OP{Emb}_{\b_i}|_{\left\{\frac{r_{\b}}{2}\right\}\times \B S^1}$ is the curve $\b_i$. The area form on the annuli $(0,r_{\a})\times \B S^1$ and $(0,r_{\b})\times \B S^1$ is the standard Euclidean area. In addition, we require that for each $i\neq j$ the image of $\OP{Emb}_{\a_i}$ is disjoint from the images of $\OP{Emb}_{\a_j}$ and $\OP{Emb}_{\b_j}$, and the image of $\OP{Emb}_{\b_i}$ is disjoint from the images of $\OP{Emb}_{\a_j}$ and $\OP{Emb}_{\b_j}$.

It is straightforward to construct two Hamiltonian isotopies of autonomous Hamiltonian diffeomorphisms
$$
h_t:(0,r_{\a})\times \B S^1\to (0,r_{\a})\times \B S^1\qquad g_t:(0,r_{\b})\times \B S^1\to (0,r_{\b})\times \B S^1,
$$
where $h_0=\OP{Id}=g_0$ such that:
\begin{itemize}
\item
For each $t\in \B R$ the diffeomorphism $h_t$ equals to the identity
in the neighborhood of $\{0,r_{\a}\}\times \B S^1$, the diffeomorphism $g_t$ equals to the identity
in the neighborhood of $\{0,r_{\b}\}\times \B S^1$.

\item
Each diffeomorphism $h_t$ preserves the foliation of
$(0,r_{\a})\times \B S^1$ by the circles $\{x\}\times \B S^1$, and $\frac{\partial h_t}{\partial t}$ is constant on each circle. Let $0\leq r'<r''\leq r_{\a},r_{\b}$. For
every $x\in (r',r'')$ the restriction
$h_t\colon \{x\}\times \B S^1\to \{x\}\times \B S^1$ is the rotation
by $2\pi\,t$. It follows that the time-one map $h_1$ equals the
identity on $(r',r'')\times \B S^1$.

\item
Each diffeomorphism $g_t$ preserves the foliation of
$(0,r_{\b})\times \B S^1$ by the circles $\{x\}\times \B S^1$, and $\frac{\partial g_t}{\partial t}$ is constant on each circle. For
every $x\in (r',r'')$ the restriction
$g_t\colon \{x\}\times \B S^1\to \{x\}\times \B S^1$ is the rotation
by $2\pi\,t$. It follows that the time-one map $g_1$ equals the
identity on $(r',r'')\times \B S^1$.

\item
Each diffeomorphism $h_t$ preserves the orientation on $(r'', r_{\a})\times~\B S^1$ and $0<\frac{\partial h_t}{\partial t}<2\pi$ in this region.
Each diffeomorphism $g_t$ preserves the orientation on $(r'', r_{\b})\times \B S^1$ and $0<\frac{\partial g_t}{\partial t}<2\pi$ in this region.
\end{itemize}

We identify $(0,r_{\a})\times \B S^1$ with its image
with respect to the embedding $\OP{Emb}_{\a_i}$, and $(0,r_{\b})\times \B S^1$ with its image
with respect to the embedding $\OP{Emb}_{\b_i}$. Then we extend each isotopy $h_t$ and $g_t$ by the identity on
$\B\Sigma_g\setminus (0,r_\a)\times \B S^1$ and on $\B\Sigma_g\setminus (0,r_\b)\times \B S^1$ obtaining smooth isotopies
$h_{t,i},g_{t,i}\in \Ham(\Sigma_g)$. The images of $\OP{Emb}_{\a_i}$ and $\OP{Emb}_{\b_i}$ are shown in Figure \ref{fig:fig8-diff}.

Recall that $f_i:=w(h_{1,i},g_{1,i})$. Let us compute $\overline{\Psi}_i(f_j)$. We have
\begin{eqnarray*}
\overline{\Psi}_i(f_j)&=&\lim\limits_{p\to\infty}\frac{1}{p}\int\limits_{W_{h,j}} \phi_i([w^p(h_{1,j},g_{1,j})_x])\o
+\lim\limits_{p\to\infty}\frac{1}{p}\int\limits_{W_{g,j}} \phi_i([w^p(h_{1,j},g_{1,j})_x])\o\\
&+&\lim\limits_{p\to\infty}\frac{1}{p}\int\limits_{V_{h,j}} \phi_i([w^p(h_{1,j},g_{1,j})_x])\o
+\lim\limits_{p\to\infty}\frac{1}{p}\int\limits_{V_{g,j}} \phi_i([w^p(h_{1,j},g_{1,j})_x])\o\\
&+&\lim\limits_{p\to\infty}\frac{1}{p}\int\limits_{U_{h,j}\setminus (V_{g,j}\cup U_{g,j})} \phi_i([w^p(h_{1,j},g_{1,j})_x])\o\\
&+&\lim\limits_{p\to\infty}\frac{1}{p}\int\limits_{U_{g,j}\setminus (V_{h,j}\cup U_{h,j})} \phi_i([w^p(h_{1,j},g_{1,j})_x])\o\\
&+&\lim\limits_{p\to\infty}\frac{1}{p}\int\limits_{U_{h,j}\cap U_{g,j}} \phi_i([w^p(h_{1,j},g_{1,j})_x])\o\thinspace .
\end{eqnarray*}
Note that if $x\in W_{h,j}$ or $x\in U_{h,j}\setminus (V_{g,j}\cup U_{g,j})$ then $[w^p(h_{1,j},g_{1,j})_x]$ equals to $[\a_j]^p$, hence $$\phi_i([w^p(h_{1,j},g_{1,j})_x])=0.$$
Similarly, if $x\in W_{g,j}$ or $x\in U_{g,j}\setminus (V_{h,j}\cup U_{h,j})$ then $[w^p(h_{1,j},g_{1,j})_x]$ is equals to $[\b_j]^p$, hence $$\phi_i([w^p(h_{1,j},g_{1,j})_x])=0.$$
In addition, if $x\in (U_{h,j}\cap U_{g,j})$ then $[w^p(h_{1,i},g_{1,i})_x]$ equals to $[\g_j]^p$, hence
$$\phi_i([w^p(h_{1,i},g_{1,i})_x])=p\textrm{ and }\phi_i([w^p(h_{1,j},g_{1,j})_x])=0\textrm{ for } i\neq j.$$
It follows that
\begin{eqnarray*}
\overline{\Psi}_i(f_i)&=&\lim\limits_{p\to\infty}\frac{1}{p}\int\limits_{V_{h,i}\cup V_{h,i}} \phi_i([w^p(h_{1,i},g_{1,i})_x])\o
+\area (U_{h,i}\cap U_{g,i})\\
\overline{\Psi}_i(f_j)&=&\lim\limits_{p\to\infty}\frac{1}{p}\int\limits_{V_{h,j}\cup V_{h,j}} \phi_i([w^p(h_{1,j},g_{1,j})_x])\o,
\rm{ where }\thinspace\thinspace \textit{i}\neq \textit{j}\thinspace .
\end{eqnarray*}
\begin{figure}[htb]
\centerline{\includegraphics[height=2.8in]{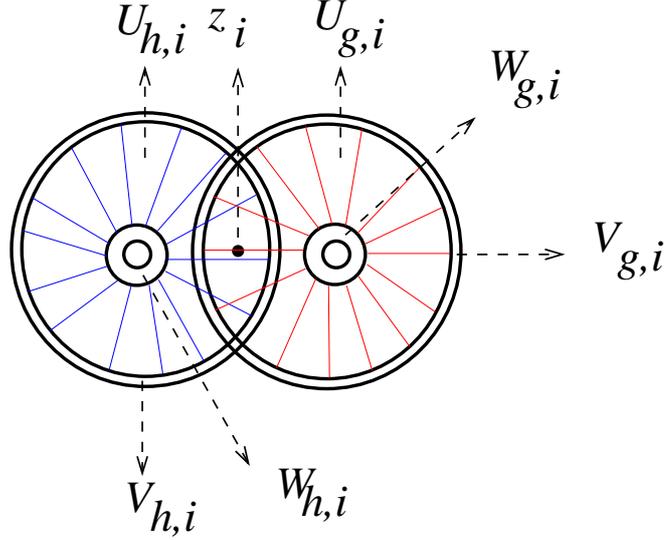}}
\caption{\label{fig:fig8-diff} Images of $\OP{Emb}_{\a_i}$ and $\OP{Emb}_{\b_i}$ are shown on the left and right respectively. The regions $W_{h,i}=\OP{Emb}_{\a_i}(0,r')$, $U_{h,i}=\OP{Emb}_{\a_i}(r',r'')$, $V_{h,i}=\OP{Emb}_{\a_i}(r'', r_\a)$, and similarly $W_{g,i}=\OP{Emb}_{\b_i}(0,r')$, $U_{g,i}=\OP{Emb}_{\b_i}(r',r'')$, $V_{g,i}=\OP{Emb}_{\b_i}(r'', r_\b)$.}
\end{figure}
For $[\g]\in \Im(\emb_j)_*$ denote by $l_j([\g])$ the length of $[\g]$ in generators $[\a_j],[b_j]$. Since $\phi_i$ is a homogeneous quasi-morphism and in addition $\phi_i([\a_j])=\phi_i([\b_j])=0$, we have
$$|\phi_i([\g])|\leq D_{\phi_i}l_j([\g]).$$

Let $x\in V_{h,j}\cup V_{h,j}$, then from the construction of $h_{t,j},g_{t,j}$ we have
$$l_j([w^p(h_{1,j},g_{1,j})_x])\leq p\cdot l_j([\g_j]).$$
Denote by $L:=\max\limits_j{l_j([\g_j])}$, $A_V:=\max\limits_j{\area(V_{h,j}\cup V_{h,j})}$, $\mathfrak{D}_\phi:=\max\limits_j{D_{\phi_j}}$.  It follows that
\begin{equation*}
|\overline{\Psi}_i(f_i)-\area (U_{h,i}\cap U_{g,i})|\leq A_V\cdot L\cdot\mathfrak{D}_\phi\quad\textrm{and}\quad
|\overline{\Psi}_i(f_j)|\leq A_V\cdot L\cdot\mathfrak{D}_\phi\thinspace .
\end{equation*}
Note that we can choose positive numbers $r_\a,r_\b,r',r''$ and $A>0$ such that $A=\area (U_{h,i}\cap U_{g,i})$ for each $i$, and
$A_V\cdot L\cdot\mathfrak{D}_\phi\leq\varepsilon A$. Now, quasi-morphisms $\frac{\overline{\Psi}_i}{A}$ satisfy the required properties.

\subsubsection{Continuation of the proof.}
\label{subsec-proof-cont}
It follows from the properties of $h_{t,i}$ and $g_{t,i}$ that diffeomorphisms $f_i$ and $f_j$ commute for $1\leq i,j\leq 2g-2$, and the $(2g-2)\times(2g-2)$ matrix, whose $ij$-th entry is $\overline{\Psi}_i(f_j)$, is non-singular. Hence there exists a family $\{\Omega_i\}_{i=1}^{2g-2}$ of genuine homogeneous quasi-morphisms on $\Ham(\B\Sigma_g)$ such that $\Omega_i(f_j)=\delta_{ij}$. Let
$$
I\colon\B Z^{2g-2}\to\Ham(\B\Sigma_g)
$$
where
$$I(d_1,\ldots,d_{2g-2})=f_1^{d_1}\circ\ldots\circ f_{2g-2}^{d_{2g-2}}.$$
It is obvious that $I$ is a monomorphism. In
order to complete the proof of the theorem in case of surfaces, it is left to show that $I$ is a
bi-Lipschitz map. Since all $f_i$ commute and $\Omega_i(f_j)=\delta_{ij}$, we obtain
$$
\|f_1^{d_1}\circ\ldots\circ f_{2g-2}^{d_{2g-2}}\|_{\OP{Frag}}\geq
\frac{|\Omega_i(f_1^{d_1}\circ\ldots\circ f_{2g-2}^{d_{2g-2}})|}{D_{\Omega_i}}
=\frac{|d_i|}{D_{\Omega_i}}\thinspace,
$$
where
$D_{\Omega_i}$ is the defect of the quasi-morphism
$\Omega_i$. The defect $D_{\Omega_i}\neq~0$ because
each $\Omega_i$ is a genuine quasi-morphism.
Let $\mathfrak{D}_\Omega:=\max\limits_i D_{\Omega_i}$. We obtain the following inequality
\begin{equation*}
\|f_1^{d_1}\circ\ldots\circ f_{2g-2}^{d_{2g-2}}\|_{\OP{Frag}}
\geq((2g-2)\cdot\mathfrak{D}_\Omega)^{-1}\sum_{i=1}^{2g-2} |d_i|\thinspace .
\end{equation*}
Denote by $\mathfrak{M}_F:=\max\limits_i\|g_i\|_{\OP{Frag}}$. Now we
have the following inequality
\begin{equation*}
\|f_1^{d_1}\circ\ldots\circ f_{2g-2}^{d_{2g-2}}\|_{\OP{Frag}}
\leq \sum_{i=1}^m |d_i|\cdot\|f_i\|_{\OP{Frag}}
\leq \mathfrak{M}_F\cdot \sum_{i=1}^{2g-2} |d_i|\thinspace .
\end{equation*}
This concludes the proof of the theorem in case of surfaces.

Case \textbf{2.} Let $\dim(\B M)>2$. Let $z_1,\ldots,z_m$ different points in $\B M$. Since the dimension of $\B M$ is greater then two, and there exists a monomorphism $\B F_2\hookrightarrow \pi_1(\B M)$ such that the image of the induced map $\widehat{Q}(\pi_1(\B M))\to \widehat{Q}(\B F_2)$ is infinite dimensional, there exist $m$ disjoint embeddings of figure eight into $\B M$ such that
\begin{itemize}
\item
an image of a singular point of the figure eight under the $i$-th embedding equals to some $z_i$,
\item
the $i$-th embedding induces an injective map $\B F_2\hookrightarrow\pi_1(\B M, z_i)$, and
the image of the induced map $\widehat{Q}(\pi_1(\B M, z_i))\to \widehat{Q}(\B F_2)$ is infinite dimensional.
\end{itemize}

As in case of surfaces we can construct, for each $1\leq i\leq m$, autonomous Hamiltonian flows $h'_{t,i},g'_{t,i}\in\Ham(\B M)$, a diffeomorphism $f'_i$ which can be written as a finite product of diffeomorphisms $h'_{1,i},g'_{1,i}$ and their inverses, and homogeneous quasi-morphism $\overline{\Psi}'_i$ such that
\begin{itemize}
\item
for each $i\neq j$ the support of the flow $h'_{t,i}$ is disjoint from the supports of flows $h'_{t,j}$ and $g'_{t,j}$, and the support of the flow $g'_{t,i}$ is disjoint from the supports of flows $h'_{t,j}$ and $g'_{t,j}$.
\item
We have $|\overline{\Psi}'_i(f'_i)-1|<\varepsilon$ and $|\overline{\Psi}'_i(f'_j)|<\varepsilon$ for $i\neq j$.
\end{itemize}
The construction of these objects is very similar to one in the case of surfaces and is omitted. Now we proceed exactly as in Case \textbf{1}, and the proof follows.
\qed

\subsection{Proof of Theorem \ref{T:Calabi-qm}}
\label{ssec-proof-Calabi-qm}

Let $g>1$ and $n=2$. Recall that
$$\B B_2<\B B_2(\B\Sigma_g)<\mathcal{MCG}_g^2.$$
Note that $\B B_2$ is a cyclic group generated by Dehn twist around two punctures. Since no positive power of the Dehn twist about a simple closed curve, which bounds a disc with 2 punctures, is conjugate in $\mathcal{MCG}_g^2$ to its inverse, then by \cite[Theorem 4.2]{BBF} there exists a homogeneous quasi-morphism in $\widehat{Q}(\mathcal{MCG}_g^2)$ which is non-trivial on $\B B_2$.

The group $\B B_2(\B\Sigma_g)$ contains a non-abelian surface group, see \cite[Corollary 5.1]{FN}, and hence it is not virtually abelian. It is not a reducible subgroup of $\mathcal{MCG}_g^2$, because it is an infinite normal subgroup of $\mathcal{MCG}_g^2$, see \cite[Corollary 7.13]{Ivanov}. It follows from the proof of  \cite[Theorem 12]{BF} that the space of homogeneous quasi-morphisms on $\mathcal{MCG}_g^2$ that remain linearly independent when restricted to $\B B_2(\B\Sigma_g)$ is also infinite dimensional. Hence there are infinitely many linearly independent homogeneous quasi-morphisms on $\B B_2(\B\Sigma_g)$ which remain non-trivial when restricted to $\B B_2$. Denote this set by $\textbf{QCal}_g$. It is infinite and it does not contain non-trivial homomorphisms, since every homomorphism on $\B B_2(\B\Sigma_g)$ must be trivial on $\B B_2$, see \cite{Bel}.

Let $\varphi\in \textbf{QCal}_g$ such that $\varphi(\g)=1$, where $\g$ is a positive generator of $\B B_2$. Take a diffeomorphism
$h\in\Ham(\B \Sigma_g)$ which is supported in some disc $\B D$. It follows from the generalized Gambaudo-Ghys construction that
$$\G_{2,g}(\varphi)(h)=\C_{\B D}(h),$$
where $\C_{\B D}\colon\Ham(\B D)\to\B R$ is the Calabi homomorphism defined in \eqref{eq:Calabi-disc}. The map
$$\G_{2,g}\colon \widehat{Q}(\B B_2(\B\Sigma_g))\to \widehat{Q}(\Ham(\B \Sigma_g))$$
is injective by Corollary \ref{cor:inf-dim}. This fact together with the fact that the set $\textbf{QCal}_g\subset\widehat{Q}(\B B_2(\B\Sigma_g))$ is infinite finishes the proof of the theorem.
\qed

\section{Comparison of bi-invariant metrics on $\Ham(\B\Sigma_g)$}
\label{sec-comparision}

The most famous metric on the group $\Ham(\B M)$ of a symplectic manifold $(\B M,\o)$
is the Hofer metric, see \cite{Ho,LM}. The associated norm is defined by
$$
\|h\|_{\OP{Hofer}}:=\inf_{H_t}\int_0^1\max_{\B M} H_t - \min_{\B M} H_tdt,
$$
where $H_t$ is a normalized Hamiltonian function generating
the Hamiltonian flow $h_t$ from the identity to $h=h_1$. Let $\a_1$ be a curve which represents a generator $[\a_1]$ of $\pi_1(\B\Sigma_g)$ defined in \eqref{eq:Pi-1-presentation}. Let $h\in \Ham(\B\Sigma_g)$ be a diffeomorphism generated by a function $H\colon\B\Sigma_g\to\B R$ which satisfies the following conditions:

\begin{itemize}
\item
there exists a positive number $C$ such that $H|_{\a_1}>C$,
\item
$H|_{\a_1}$ is not constant.
\end{itemize}

It implies that all powers of $h$ are
autonomous diffeomorphisms and hence $\|h^n\|_{\OP{Aut}} =~1$ for all $n\in \B Z$.
On the other hand, $nH$ generates $h^n$ and $nH|_{\a_1}>Cn$,
hence by a result of L. Polterovich \cite{P-diameter} we have  $\|h^n\|_{\OP{Hofer}}\geq Cn$.
On the other hand, for every $\epsilon>0$, one can easily construct an autonomous diffeomorphism whose Hofer norm is less then~$\epsilon$.
This shows that the identity homomorphism between the autonomous metric and the Hofer metric is
not Lipschitz in neither direction.

In this paper we showed that there exists a Hamiltonian diffeomorphism $f$ supported in some disc such that $\|f^k\|_{\Aut}\xrightarrow[k\rightarrow\infty]{}\infty$, but $\|f^k\|_{\Frag}=~1$ for all $k$. On the other hand, one can show that there exists an autonomous diffeomorphism $g$, such that $\|g^k\|_{\Frag}\xrightarrow[k\rightarrow\infty]{}\infty$. It follows that the identity homomorphism between
the autonomous metric and the fragmentation metric is not Lipschitz in neither direction as well.

\subsection*{Acknowledgments}
The author would like to thank Mladen Bestvina, Danny Calegari, Louis Funar, Juan Gonzalez-Meneses, Jarek Kedra, Chris Leininger, Dan Margalit and Pierre Py for fruitful discussions. Part of this work has been done during the author's stay in Mathematisches For\-schung\-sinstitut Oberwolfach and in Max Planck Institute for Mathematics in Bonn. The author wishes to express his gratitude to both institutes. He was supported by the Oberwolfach Leibniz fellowship and Max Planck Institute research grant.

\bigskip

Max-Planck-Institut f$\ddot{\textrm{u}}$r Mathematik, 53111 Bonn, Germany\\
\emph{E-mail address:} \verb"brandem@mpim-bonn.mpg.de"

\end{document}